\documentclass{amsart}
\usepackage{amsmath,amssymb,bbm,color,morefloats}
\usepackage[english, algosection, algoruled, noline]{algorithm2e}

\usepackage{multirow}
\newcommand{\ignore}[1]{}
\newcommand{\piFB}{\pi_{F,\BB}}

\newcommand{\xx}{\mathbf x}
\newcommand{\BB}{\mathcal{B}}
\newcommand{\CC}{\mathcal{C}}

\newcommand{\XX}{\mathcal{X}}

\newcommand{\ub}{\mathbf u}
\newcommand{\PR}{\oR[\xx]}
\newcommand{\PC}{\oC[\xx]}
\newcommand{\PK}{\oK[\xx]}
\newcommand{\QPK}{\PK/I}
\newcommand{\oR}{\mathbb R}
\newcommand{\oK}{\mathbb K}
\newcommand{\clK}{\overline{\mathbb K}}

\newcommand{\oN}{\mathbb N}
\newcommand{\oC}{\mathbb C}
\newcommand{\oL}{\mathbb L}
\newcommand{\dd}{\mathbf d}

\newcommand{\Kc}{\mathcal{K}}
\newcommand{\Lc}{\mathcal L}
\newcommand{\Mon}{\mathcal{M}}
\newcommand{\Fred}{F^{\mathrm{red}}}
\newcommand{\Gred}{G^{\mathrm{red}}}

\newcommand{\Span}[1]{\<{#1}\>} 

\newcommand{\SpanK}[1]{\Span{#1}}
 
\newcommand{\SpanD}[2]{\<#1\,|\,{ #2}\>}
\newcommand{\Ker}{\tmop{Ker}}
\newcommand{\Trace}{\tmop{Trace}}

\newcommand{\rank}{\tmop{rank}}

\newcommand{\AAA}{\mathcal{A}}
\newcommand{\HE}{H_{\Lambda}^E}
\newcommand{\KE}{\Ker H_{\Lambda}^{E}}
\newcommand{\KerE}[1]{\Kc_{E}(#1)}

\newcommand{\un}{\mathbf{1}}
\newcommand{\unb}{\underline{\mathbf{1}}}
\newcommand{\supp}{\mathrm{supp}}

\newcommand\note[1]{}
\newcommand\noteBM[1]{}
\newcommand\noteML[1]{}
\newcommand\comment[1]{} 

\def\<{\langle}
\def\>{\rangle}
\def\Mon{\mathcal{M}}

\newcommand{\support}{\textrm{support}}

\newcommand{\tmmathbf}[1]{\ensuremath{\boldsymbol{#1}}}
\newcommand{\tmop}[1]{\ensuremath{\operatorname{#1}}}

\newenvironment{itemizeminus}{\begin{itemize} }{\end{itemize}}

\newtheorem{definition}{Definition}[section]
\newtheorem{corollary}[definition]{Corollary}
\newtheorem{theorem}[definition]{Theorem}

\newtheorem{lemma}[definition]{Lemma}
\newtheorem{proposition}[definition]{Proposition}
\newtheorem{remark}[definition]{Remark}

\def\LL{\mathbb{L}}
\def\NN{\mathbb{N}}

\begin{document}

\title{Moment matrices, border bases and real radical computation}
\author{J.B. Lasserre\and M. Laurent\and B. Mourrain \and
  Ph. Rostalski \and Ph. Tr\'ebuchet}
\address{J.B. Lasserre, LAAS, Toulouse, France}
\address{Monique Laurent, CWI, Amsterdam, Netherland}
\address{Bernard Mourrain, GALAAD INRIA, Sophia Antipolis, France}
\address{Philipp Rolstalki, University of Berkeley, USA}
\address{Philippe Tr\'ebuchet, LIP6, University Paris VI, France}

\begin{abstract}
In this paper, we describe new methods to compute the radical
(resp. real radical) of an ideal, assuming it complex (resp. real)
variety is finte. The aim is to combine approaches for solving a system of 
polynomial equations with dual methods which involve moment matrices and
semi-definite programming.  While the border basis algorithms of \cite{Mourrain2005}
are efficient and numerically stable for computing complex roots, algorithms
based on moment matrices \cite{LLR07} allow the incorporation of additional
polynomials, e.g., to restrict the computation to real roots or to eliminate
multiple solutions. The proposed algorithm can be used to compute a border
basis of the input ideal and, as opposed to other approaches, it can also
compute the quotient structure of the (real) radical ideal directly,
i.e., without prior algebraic techniques such as Gr\"obner bases. It thus
combines the strength of existing algorithms and provides a unified treatment
for the computation of border bases for the ideal, the radical ideal and the
real radical ideal.
\end{abstract}
\maketitle

\section{Introduction}
Many problems in mathematics and science can be reduced to the task
of solving zero-dimensional systems of polynomials. Existing methods
for this task often compute all (real and complex) roots. However,
often only real solutions are significant and one needs to sieve out
all complex solutions afterwards in a separate step.

Typical approaches in this vein are the efficient homotopy
continuation methods in the spirit of \cite{Ve99}, \cite{SW05},
recursive intersection techniques using rational univariate
representation \cite{Giusti2001} in the spirit of 
Kronecker's work \cite{Kronecker1882}, Gr\"obner basis approaches
using eigenvector computations or rational univariate representation
\cite{CLO98}, \cite{Rou99}, \cite[chap. 4]{em-07-irsea}. In the latter
methods, emphasis is put on exact input and computation. Using a
different approach, Mourrain and Tr\'ebuchet \cite{Mourrain2005} have
proposed an efficient numerical algorithm that uses border bases and
the concept of {\it rewriting family}. In particular, in the course of
this algorithm, a distinguishing and remarkable feature is a careful
selection strategy for monomials serving as candidates for elements in
a basis of the quotient space $\oK[\xx]/I$ (if $I\subset\oK[\xx]$ is
the ideal generated by the polynomials defining the equations). As a
result, at each iteration of the procedure, the candidate basis for
the quotient space $\oK[\xx]/I$ contains only a small number of
monomials (those associated with a certain {\it rewriting
  family}). Another nice feature of this approach (and in contrast
with Gr\"obner base approaches) is its robustness with respect to
perturbation of coefficients in the original system.

On the other hand, Lasserre et al. \cite{LLR07} have proposed an
alternative numerical method, real algebraic in nature, to directly
compute all real zeros {\it without} computing any complex zero.
This approach uses well established semi-definite programming
techniques and numerical linear algebra. Remarkably, all information
needed is contained in the so-called quasi-Hankel {\it moment
matrix} with rows and columns indexed by the canonical monomial
basis of $\oK[\xx]_d$. Its entries depend on the polynomials
generating the ideal $I$ and the underlying geometry  when this
matrix is required to be positive semi-definite with maximum rank. A
drawback of this approach is the potentially large size of the
positive semi-definite moment matrices to handle in the course of the
algorithm. Indeed, when the total degree is increased from $d$ to
$d+1$, the new moment matrix to consider has its rows and columns
indexed by the canonical (monomial) basis of $\oK[\xx]_{d+1}$.

The goal of this paper is to combine a main feature of the border basis 
algorithm of \cite{Mourrain2005} (namely its careful selection of monomials,
considered as candidates in a basis of the quotient space
$\oK[\xx]/I$) with the semi-definite approach of \cite{LLR07} for
computing real zeros and an approach for computing the radical ideal
inspired by \cite{JFSMR08}.  

The main contribution of this paper is to describe a new algorithm
which incorporates in the border basis algorithm the positive
semi-definiteness constraint of the moment matrix, which are much
easier to handle than the relaxation method of \cite{LLR07}.  We show
the termination of the computation in the case where the real radical
is zero-dimensional (even in cases where the ideal is not
zero-dimensional). 
A variant of the approach is also proposed, which yields a new
algorithm to compute the (complex) radical for zero-dimensional
ideals.

In this new algorithm, the rows and columns involved in the
semi-definite programming problem are associated with the family of
monomials (candidates for being in a basis of the quotient space) and
its border, i.e., a subset of monomials much smaller than the
canonical (monomial) basis of $\PR_d$ considered in \cite{LLR07}.  As
a result, the (crucial) positive semi-definiteness constraint is much
easier to handle and solving problem instances of size much larger
than those in \cite{LLR07} can now be envisioned.
A preliminary implementation of this new algorithm
validate experimentally these improvements on few benchmarks problems.

The approach differs from previous techniques such as \cite{BeckerNeuhaus98}
which involve complex radical computation and factorisation or
reduction to univariate polynomials, in that the new polynomials
needed to describe the real radical are computed directly from the
input polynomials, using SDP techniques.

The paper is organized  as follows. Section \ref{sec:2} recalls the
ingredients and properties involved in the algebraic computation.
Section \ref{sec:3} describes duality tools and Hankel operators
involved in the computation of (real) radical of ideals.
In Section \ref{sec:4}, we analyse the properties of the truncated
Hankel operators. In section \ref{sec:5}, we describe the real radical
and radical algorithms and prove their correctness in section
\ref{sec:6}. Finally, Section \ref{sec:7} contains some illustrative examples
and experimentation results of a preliminary implementation.

\section{Polynomials, dual space and quotient algebra}\label{sec:2}
In this section, we set our notation and recall the eigenvalue techniques 
for solving polynomial equations and the border basis method. These results will be
used for showing the termination of the radical border basis algorithm. 
\subsection{Ideals and varieties}

Let $\PK$ be the set of the polynomials in the variables
$\xx=(x_1,\ldots,x_n)$, with coefficients in the field
$\oK$. Hereafter, we will choose\footnote{For notational simplicity, we will
consider only these two fields in this paper, but $\oR$ and $\oC$ can
be replaced respectively by any real closed field and any field
containing its algebraic closure)}
 $\oK=\oR$ or $\oC$.
 Let $\overline{\oK}$ denotes the algebraic closure of
 ${\oK}$.
For $\alpha \in \oN^n$, $\xx^{\alpha}= x_1^{\alpha_1} \cdots x_n^{\alpha_n}$ is the monomial with exponent $\alpha$
and degree $|\alpha|=\sum_i\alpha_i$.  The set of all monomials in $\xx$ is
denoted $\Mon = \Mon(\xx)$. We say that $\xx^{\alpha} \le \xx^{\beta}$ if
$\xx^{\alpha}$ divides $\xx^{\beta}$, i.e., if $\alpha\le \beta$ coordinate-wise.
For a polynomial $f=\sum_\alpha f_\alpha \xx^\alpha$, its support is 
$\supp(f):=\{\xx^\alpha\mid f_\alpha\ne 0\}$, the set of monomials occurring with a nonzero coefficient in $f$.

 For $t\in \NN$ and $S\subseteq \PK$, we introduce the following sets:
\begin{itemize}
 \item $S_{t}$ is the set of elements of $S$ of degree $\le t$,
 \item $S_{[t]}$ is the set of element of $S$ of degree exactly $t$,
\item $\Span{S} = \big\{ \sum_{f\in S} \lambda_{f}\, f\ |\ f\in S, \lambda_f\in \oK\big\}$ is the linear span of $S$,
\item $(S) = \big\{ \sum_{f\in S} p_f\, f \ | \ p_f \in \PK, f \in S \big\}$
is the ideal in $\oK[\xx]$ generated by $S$,
\item
$\SpanD{S}{t} = \big\{ \sum_{f\in S_t} p_f\, f\ | \ p_f \in
  \PK_{t-\deg(f)}\big\}$ is the vector space spanned by $\{\xx^\alpha  f\mid
  f\in S_t, |\alpha|\le t-\deg(f)\}$, 
\item 
$S^+:=S\cup x_1S \cup \ldots \cup x_n S$ is the prolongation of $S$ by one degree,
\item
$\partial S:= S^+\setminus S$ is the border of $S$,
\item
 $S^{[t]}:= S^{\stackrel{t\ \mathrm{times}}{+\cdots+}}$
is the result of applying $t$ times the prolongation operator `$^{+}$' on $S$, with 
$S^{[1]}=S^+$ and, by convention, $S^{[0]}=     S$.
\end{itemize}
Therefore,
$S_t=S\cap \oK[\xx]_t$, $S_{[t]}= S\cap \oK[\xx]_{[t]}$,
$S^{[t]}=\{x^\alpha f\mid f\in S, |\alpha|\le t\}$, 
$\SpanD{S}{t} \subseteq (S)\cap \oK[\xx]_t= (S)_t$, but the inclusion may be strict.
\ignore{
We use
\[
\Span{S} = \big\{ \sum_{f\in S} \lambda_{f}\, f\ |\ f\in S, \lambda_f\in \oK\big\}
\]
to denote the linear span of the polynomials in $S$ and 
\[
\Span{S}_{t} = \Span{S} \cap \PK_t 
\]
to denote the subset of polynomials of degree at most $t$. Furthermore we will need the notation
\[
(S) = \big\{ \sum_{f\in S} p_f\, f \ | \ p_f \in \PK, f \in S \big\}
\]
for the ideal generated by all polynomials in $S$ and
\begin{align*} 
\SpanD{S}{t} &= \big\{ \sum_{f\in S_t} p_f\, f\ | \ p_f \in \PK_{t-\deg(f)}\big\}\\
&=\big\{ \sum_{\alpha\in \NN^{n}, f\in S_t}\lambda_{\alpha,f}\,
\xx^{\alpha}\, f\ |\ \lambda_{\alpha,f}\in \oK \; \rm{ and } \; |\alpha|\leq t-\deg(f) \big\}
\end{align*} 
for the subset of the ideal by considering polynomial multipliers of limited
degree. This latter set is not to be confused with the set $(S)_t (S) \cap \PK_t$ of polynomials in the ideal $(S)$ with degree at most
$t$. We have $\SpanD{S}{t}\subseteq (S)_{t} $ but these two sets may or may not coincide.. 
}

If $\BB \subseteq \Mon$ contains $1$ then, for any monomial $m \in \Mon$, there exists an integer 
$k$ for which 
 $m\in \BB^{[k]}$. The {\em $\BB$-index} of $m$, denoted by $\delta_{\BB}(m)$,
 is defined as the smallest integer $k$ for which  $m\in \BB^{[k]}$.

A set of monomials $\BB$ is said to be {\em connected to $1$} if
$1\in \BB$ and for every monomial $m\neq 1$ in $\BB$, $m= x_{i_{0}} m'$ for
some $i_{0}\in [1,n]$ and $m'\in \BB$. 

Given a vector space $E \subseteq \PK$, its prolongation
$ E^+ : = E + x_1 E + \ldots + x_n E$ is again a vector space.

The vector space $E$ is said to be \textit{connected to 1} if 
$1\in E$ and any
non-constant polynomial $p \in E$ can be written as $p = p_0+ \sum_{i = 1}^n x_i p_i$
for some polynomials $p_0,p_i\in E$ with $\deg(p_0)\le \deg(p)$, $\deg(p_i)\le \deg(p)-1$ for $i\in [1,n]$.
Obviously,
$E$ is connected to 1 when $E=\Span{\CC}$ for some monomial set $\CC\subseteq \Mon$ which is connected to 1. 
Moreover, $E^+ = \Span{\CC^+}$ if $E=\Span{\CC}$.

\medskip
Given an ideal $I\subseteq \PK$ and a field $\LL \supseteq \oK$, we denote by
\begin{align*}
V_{\LL}(I)&:=\{x\in\LL^n\mid f(x)=0\ \forall f\in I\}
\end{align*}
its associated variety in $\LL^{n}$. By convention $V(I)=V_{\clK}(I)$. For a set
$V\subseteq \oK^n$, we define its vanishing ideal
\[I(V):=\{f\in \PK\mid f(v)=0\ \forall v\in V\}.\]
Furthermore, we denote by
\[\sqrt I:=\{f\in \PK \,\mid\, f^m\in I \ \text{ for some } m\in \oN\setminus \{0\}\}\]
the radical of $I$. 

For $\oK=\oR$,  
we have $V(I)=V_{\oC}(I)$, 
but one may also be interested in the subset of real solutions, namely the real variety
$V_\oR(I)=V(I)\cap \oR^n.$
The corresponding vanishing ideal is $I(V_\oR(I))$ and the \emph{real radical ideal} is
\[\sqrt[\oR]I:=\{p\in \PR\mid p^{2m} +\sum_j q_j^2\in I \ \text{ for some }q_j\in \PR, m\in \oN\setminus \{0\}\}.\]
Obviously,
\[I\subseteq \sqrt I\subseteq I(V_{\oC}(I)),\ \  I\subseteq \sqrt[\oR]I\subseteq I(V_\oR(I)).\]
An ideal $I$ is said to be {\em radical} (resp., {\em real radical})
if $I=\sqrt I$ (resp. $I= \sqrt[\oR]I$). Obviously, $I\subseteq
I(V(I))\subseteq I(V_\oR(I))$. Hence, if $I\subseteq \PR$ is real
radical, then $I$ is radical and moreover, $V(I)=V_\oR(I)\subseteq
\oR^n$ if $|V_\oR(I)|<\infty$.

The following two famous theorems relate vanishing and radical
ideals:
\begin{theorem}\ 
\begin{itemize}
\item[(i)] {\bf Hilbert's Nullstellensatz} (see, e.g., \cite[\S 4.1]{CLO97})
  $\sqrt I=I(V_{\oC}(I))$ for an ideal $I\subseteq \PC$.
\item[(ii)] {\bf Real Nullstellensatz} (see, e.g., \cite[\S 4.1]{BCR98})
$ \sqrt[\oR]I=I(V_\oR(I))$ for  an  ideal $I\subseteq \PR$.
\end{itemize}\end{theorem}

\subsection{The quotient algebra}
Given an ideal $I\subseteq \PK$, the quotient set $\QPK$ consists of all
cosets $[f]:=f+I=\{f+q \mid q \,\in\,I\}$ for $f\in \PK$, i.e., all equivalent
classes of polynomials of $\PK$ modulo the ideal $I$.  The quotient set
$\QPK$ is an algebra with addition $[f]+[g]:=[f+g]$, scalar multiplication
$\lambda [f]:=[\lambda f]$ and with multiplication $[f][g]:=[fg]$, for
$\lambda\in \oR$, $f,g\in \PK$.

A useful property is that, when $I$ is zero-dimensional
(i.e., $|V_{\clK}(I)|<\infty$) then
$\QPK$ is a finite-dimensional vector space and  its dimension is related to
the cardinality of $V(I)$, as indicated in Theorem \ref{theodim} below.
\begin{theorem}\label{theodim}
Let $I$ be an ideal in $\PK$. Then 
$|V_{\clK}(I)|<\infty \Longleftrightarrow
\dim \QPK<\infty.$
Moreover, $|V_{\clK}(I)|\le \dim\ \QPK$, with equality if and only if $I$ is radical.
\end{theorem}

A proof of this theorem and a detailed treatment of the quotient algebra
$\QPK$ can be found e.g., in \cite{CLO97}, \cite{em-07-irsea}, \cite{St04}.

Assume $|V_{\clK}(I)|<\infty$ and set $N:=\dim \QPK$, $|V_{\clK}(I)|\le N<\infty$. Consider a set 
 $\BB:=\{b_1,\ldots,b_N\}\subseteq \PK$ for which 
$\{[b_1],\ldots,[b_N]\}$ is a basis of $\QPK$; by abuse of language we also say that $\BB$ itself is a basis of
$\QPK$. Then every $f\in \PK$ can be written in a unique way as $f=\sum_{i=1}^N c_i b_i +p,$ where $c_i\in \oK,$ $ p\in I;$
the polynomial 
$\pi_{I,\BB}(f):=\sum_{i=1}^N c_i b_i$ is called the remainder of $f$ modulo
$I$, or its {\em normal form}, with respect to the basis $\BB$. In other
words, $\Span{\BB}$ and $\QPK$ are isomorphic vector spaces. 

\subsubsection{Multiplication operators}
Given a polynomial $h\in \PK$, we can define the {\em multiplication (by $h$) operator} as
\begin{equation}
\label{mult}
\begin{array}{lccl}
\XX_h: & \QPK & \longrightarrow & \QPK\\
     &  [f] & \longmapsto &\XX_h([f])\,:=\, [hf]\,,
\end{array}
\end{equation}
Assume that $N:=\dim \QPK<\infty$. Then the multiplication operator $\XX_h$ can be represented by its matrix, again denoted $\XX_h$ for simplicity, with respect to a given basis 
$\BB=\{b_1,\ldots,b_N\}$ of $\QPK$. 

Namely, setting $\pi_{I,\BB}(hb_j):= \sum_{i=1}^N a_{ij}b_i$ for some scalars $a_{ij}\in \oK$,
 the $j$th column of $\XX_h$ is the vector $(a_{ij})_{i=1}^N$.
Define the vector
$\zeta_{\BB,v}:=(b_j(v))_{j=1}^N \in \clK^N$, whose coordinates are the
evaluations of the polynomials $b_j \in \BB$ at the point $v\in \clK^n$. The
following famous result (see e.g., \cite[Chapter 2\S4]{CLO98}, \cite{em-07-irsea}) relates the eigenvalues of the multiplication operators in $\QPK$ to the algebraic variety $V(I)$.
This result underlies the so-called eigenvalue method for solving polynomial equations and plays a central role in many algorithms, also in the present paper.

\begin{theorem}
\label{thm::muloperator} Let $I$ be a zero-dimensional ideal in $\PK$, $\BB$ a basis of $\QPK$, and $h\in \PK$.
The eigenvalues of the multiplication operator $\XX_h$ 
are the evaluations $h(v)$ of the polynomial $h$ at the points $v \in V(I)$.
Moreover, $(\XX_h)^T\zeta_{\BB,v}=h(v) \zeta_{\BB,v}$ and
the set of common eigenvectors of $(\XX_h)_{h\in \PK}$ are up to a
non-zero scalar multiple the
vectors $\zeta_{\BB,v}$ for $v\in V(I)$.
\end{theorem}

Throughout the paper we also denote by $\XX_i:=\XX_{x_i}$ the matrix of the multiplication operator by the variable $x_i$. 
By the above theorem, the eigenvalues of the matrices $\XX_{i}$ are the $i$th coordinates of the points $v\in V(I)$. 
Thus the task of solving a system of polynomial equations is reduced to a task of numerical linear algebra once a basis of $\QPK$ and a normal form algorithm are available, permitting the construction of the multiplication matrices $\XX_i$.


\subsection{Border bases}
The eigenvalue method for solving polynomial equations from the above section  requires
the knowledge of a basis of $\QPK$ and an algorithm to compute the normal form of
a polynomial with respect to this basis. In this section we will recall a
general method for obtaining such a basis and a method to reduce polynomials
to their normal form. 

\ignore{
For any subset $S\subset\PK$, we denote by $S^+$ the set $S^{+}= S \cup x_{1}\, S \cup \cdots \cup
x_{n}\, S$, $\partial S= S^{+} \backslash S$.  
The set $S^{+}$ is called the {\em prolongation} of $S$.
For any $k\in \NN$, $S^{[k]}$ is $S^{\stackrel{k\ \mathrm{times}}{+\cdots+}}$
the result of applying $k$ times the operator $^{+}$ on $S$. By convention, $S^{[0]}$ is $S$.

If $\BB \subseteq \Mon$ contains $1$ then, for any monomial $m \in \Mon$, there exists an integer 
$k$ for which 
 $m\in \BB^{[k]}$. The {\em $\BB$-index} of $m$, denoted by by $\delta_{\BB}(m)$,
 is the smallest integer $k$ for which  $m\in \BB^{[k]}$.

A set of monomials $\BB$ is said to be {\em connected to $1$} if, 
for every monomial $m$ in $\BB$, either
$m=1$, or $m= x_{i_{0}} m'$ with $i_{0}\in [1,n]$ and $m'\in \BB$.
}

Throughout $\BB\subseteq \Mon$ is a finite set of monomials.

\begin{definition}
A rewriting family $F$ for a (monomial) set $\BB$ is a set of
polynomials  $F=\{f_i\}_{i\in \mathcal{I}}$ such that 
\begin{itemize}
\item $\supp(f_i)\subseteq \BB^+$,
\item $f_i$ has exactly {\bf one} monomial in $\partial \BB$, denoted as 
$\gamma(f_i)$  and 
called the {\it leading monomial} of $f_i$. (The polynomial $f_i$ is
normalized so that the coefficient of $\gamma(f_i)$ is $1$.)
\item  if  $\gamma(f_i)=\gamma(f_j)$ then $i=j$.
\end{itemize}
\end{definition}

\begin{definition}
We say that the rewriting family $F$ is \emph{graded} if $\deg (\gamma(f)) =\deg
(f)$ for all $f\in F$.
\end{definition}
 
\begin{definition}\label{defnormfam}
A rewriting family $F$ for $\BB$ is said to be \emph{complete} in degree $t$ if
it is  graded and satisfies  $(\partial \BB)_t \subseteq \gamma(F)$; that is,
each monomial $ m \in \partial \BB$
of  degree at most $t$ is the leading monomial of some (necessarily unique)
 $f\in F$.
\end{definition}

\begin{definition} \label{defpiBF}
Let $F$  be  a rewriting family for $\BB$, {complete} in degree $t$.
Let $\pi_{F,\BB}$ be the projection on $\Span{\BB}$ along $F$ defined recursively on the
monomials $m \in \Mon_{t}$ in the following way:
\begin{itemize}
  \item if $m \in \BB_{t}$, then $\pi_{F,\BB}(m)=m$,
 \item if $m \in (\partial \BB)_{t} \ (= (\BB^{[1]}\setminus \BB^{[0]})_{t})$,
then  $\pi_{F,\BB}(m)=m-f$, where $f$ is the (unique) polynomial in $F$ for which
 $\gamma(f)=m$,
 \item if $m \in (\BB^{[k]}\setminus \BB^{[k-1]})_{t}$ for some integer $k\ge 2$, write $m = x_{i_{0}} m'$, where 
$m' \in \BB^{[k-1]}$ and $i_{0}\in [1,n]$ is the smallest possible variable index for which such a decomposition exists,
then 
   $\pi_{F,\BB}(m)=\pi_{F,\BB}(x_{i_{0}}\,\pi_{F,\BB}(m') )$. 
\end{itemize}
\end{definition}
One can easily verify that $\deg(\piFB(m))\le \deg(m)$ for $m\in\Mon_t$. The map $\piFB$ extends by linearity to a linear map from 
$\oK[\xx]_t$ onto $\Span{\BB}_t$. By construction, 
$f=\gamma(f)- \piFB(\gamma(f))$ and $\piFB(f)=0$ for all $f\in F_t$. 
The next theorems show that, under some natural commutativity condition,
the map $\piFB$ coincides with the linear projection from $\oK[\xx]_t$ onto $\Span{\BB}_{t}$
along the vector space $\SpanD{F}{t}$, and they introduce the notion of border bases.

\begin{definition} \label{def:borderbasis}
Let $\BB\subset \Mon$ be connected to $1$. 
A family $F\subset \PK$ is a border basis for $\BB$ if it is  a rewriting family
for $\BB$, {complete} in all degrees, and such that $\PK= \Span{\BB} \oplus (F).$
\end{definition}
\ignore{
So far we did not discuss how to decompose a monomial $m\in (\BB^{[k]}\setminus \BB^{[k-1]})_{t}$ into a monomial $m'\in \BB^{[k-1]})_{t}$ such that $m = x_{i_{0}} m'$ in the previous definition and the outcome may actually depend on this selection. The following theorem addresses this issue and defines the concept of border bases:
}

An algorithmic way to check that we have a border basis is based on
the following result, that we recall from \cite{Mourrain2005}:
\begin{theorem}\label{thmcom}
Assume that $\BB$ is connected to $1$ and let $F$ be a rewriting family for
 $\BB$, complete in degree $t\in \oN$.
Suppose that, for all $m\in \Mon_{t-2}$,
\begin{equation}\label{eqcom}
   \pi_{F,\BB}(x_i\,\pi_{F,\BB}(x_j\,m))=\pi_{F,\BB}(x_j\,\pi_{F,\BB}(x_i\,m))
\ \text{ for all } i,j\in[1,n].
\end{equation}
Then $\pi_{F,\BB}$ coincides with the linear projection of $\PK_{t}$ on $\Span{\BB}_{t}$
along the vector space $\SpanD{F}{t}$; that is,
$\PK_t=\Span{\BB}_t\oplus \SpanD{F}{t}.$
\end{theorem}
\begin{proof} 
Equation~\eqref{eqcom} implies that any choice of decomposition of $m\in
\Mon_{t}$ as a product of variables yields the same result after applying
$\pi_{F,\BB}$. Indeed, let $m=x_{i_{1}}\,m' =x_{i_{2}}\, m''$  
 with $i_{1}\neq i_{2}$ and $m', m'' \in \Mon_{t-1}$. Then
there exists $m''' \in \Mon_{t-2}$ such
that $m'= x_{i_{2}} \, m'''$, $m''= x_{i_{1}}\, m'''$. By the relation \eqref{eqcom}
we have:
\begin{eqnarray*}
\lefteqn{\pi_{F,\BB} (x_{i_{1}}\, \pi_{F,\BB}(m')) }\\
 & = & \pi_{F,\BB}(x_{i_{1}}\, \pi_{F,\BB}(x_{i_{2}} m''')) 
 =  \pi_{F,\BB}(x_{i_{2}}\, \pi_{F,\BB}(x_{i_{1}} m''')) \\
 & = & \pi_{F,\BB} (x_{i_{2}}\, \pi_{F,\BB}(m'')).
\end{eqnarray*}
Let us prove by induction on $l=\deg(m)$
 that for a monomial
$m= x_{i_{1}} \cdots x_{i_{l}}\in \Mon_t$, 
\begin{equation}\label{eqFB}
\pi_{F,\BB}(m) = \pi_{F,\BB}(x_{i_{1}} \pi_{F,\BB}(x_{i_{1}} \cdots \pi_{F,\BB}(x_{i_{l}})\cdots),
\end{equation}
does not depend on the order in which we take the monomials in 
the decomposition $m= x_{i_{1}} \cdots x_{i_{l}}$:
\begin{itemize}
 \item Either $m\in \BB$.  As $\BB$ is connected to $1$, there exists $i'\in [1,n]$
   and $m' \in \BB_{t-1}$ such that $m = \piFB(m) = \piFB(x_{i'}m')=
  \piFB(x_{i'}\,
   \piFB(m'))$, from which we deduce \eqref{eqFB}
using the  induction hypothesis applied to $m'$ and relation \eqref{eqcom}. 
 \item Or $m\not\in \BB$. Then,  by definition of $\piFB$, there exists $i'\in [1,n]$
   and $m' \in \Mon_{t-1}$ such that 
$\piFB(m) =  \piFB(x_{i'}\, \piFB(m'))$, from which we deduce \eqref{eqFB} in a similar way
using the induction hypothesis applied to $m'$ and relation \eqref{eqcom}. 
\end{itemize}
The map $\piFB$ defines a projection of $\PK_{t}$ on $\Span{\BB}_{t}$. 
It suffices now to show that $\Ker \piFB = \SpanD{F}{t}$.
First we show  
that $m-\piFB(m)\in \SpanD{F}{s}$ for
all $m\in \Mon_s$, using induction on  $s=0,\ldots,t$. The base case $s=0$ is obvious; indeed $\piFB(1)=1$ since $1\in \BB$,  and $0\in \SpanD{F}{0}$.
Consider $m\in \Mon_{s+1}$. Write $m=x_{i_0}m'$ where $m'\in\Mon_s$ and 
$\piFB(m)=\piFB(x_{i_0} \piFB(m'))$ (recall Definition \ref{defpiBF}). 
We have:
$$m-\piFB(m) = \underbrace{x_{i_0}(m'-\piFB(m'))}_{:=q} 
+ \underbrace{x_{i_0}\piFB(m')- \piFB(x_{i_0} \piFB(m'))}_{:=r}.$$
By the induction assumption, $m'-\piFB(m') \in \SpanD{F}{s}$ and thus 
$q\in \SpanD{F}{s+1}$.
Write $\piFB(m')=\sum_{b\in \BB_s}\lambda_b b$ ($\lambda_b\in\oK$). Then,
$r= \sum_{b\in \BB_s}\lambda_b (x_{i_0}b-\piFB(x_{i_0}b))$, where 
$x_{i_0}b-\piFB(x_{i_0}b)=0$ if $x_{i_0}b\in\BB$, and 
$x_{i_0}b-\piFB(x_{i_0}b) $ is a polynomial of $F_{s+1}$ otherwise. This implies
$r\in  \SpanD{F}{s+1}$ and thus $m-\piFB(m)\in \SpanD{F}{s+1}$.
Thus we have shown that
$\oK[\xx]_t =\Span{\BB}_t + \SpanD{F}{t}.$
Next, observe that $\SpanD{F}{t}\subseteq \Ker\piFB$, which follows from the fact that 
$F_t\subseteq \Ker\piFB$ together with (\ref{eqFB}).
This implies  that $\Span{\BB}_t \cap \SpanD{F}{t}=\{0\}$ and thus the equality
$\SpanD{F}{t}= \Ker\piFB$.
\end{proof}
 
In order to have a simple test and effective way to test the
commutation relations \eqref{eqcom},
we introduce now the commutation polynomials.
\begin{definition}
Let $F$ be a rewriting family and $f,f'\in F$.
Let $m,m'$ be the smallest degree monomials for which 
$m\,\gamma(f)=m'\, \gamma(f')$.
Then the polynomial 
$C(f,f'):= m f-m'f' = m' \piFB(f')-m\piFB(f)$ is called the \emph{commutation  polynomial} of $f,f'$.
\end{definition}
 
\begin{definition}
For a rewriting family $F$ with respecet to $\BB$, we denote by $C^{+}(F)$ the set of polynomials
of the form $m\, f - m'\, f',$ where $f,f'\in F$ and 
$m,m' \in \{0,1,x_{1},\ldots,x_{n}\}$ satisfy 
\begin{itemize}
 \item either $m\, \gamma(f)=m'\, \gamma(f')$,
 \item or  $m\, \gamma(f)\in \BB$ and $m'=0$.
\end{itemize}
\end{definition}
Therefore, $C^{+}(F) \subset \Span{\BB^{+}}$ and $C^+(F)$ contains all commutation
polynomials $C(f,f')$ for $f,f'\in F$ whose monomial multipliers $m, m'$ are
of degree $\leq 1$.
The next result can be deduced using Theorem \ref{thmcom}.

\begin{theorem}\label{thmnfdegd}
Let $\BB\subset \Mon$ be connected to $1$ and let $F$ be a rewriting family for $\BB$, complete in degree $t$. 
If for all $c\in C^{+}(F)$ of degree $\le t$, $\pi_{F,\BB}(c) = 0$, 
then $\pi_{F,\BB}$  is the projection of $\PK_{t}$  on $\Span{\BB}_{t}$ along
$\SpanD{F}{t}$, ie. $\PK_{t}=\Span{\BB}_{t}\oplus \SpanD{F}{t}$. 
\end{theorem}
\begin{proof} 
Let us prove by induction on $t$ that if $F$ is complete in degree $t$
and for all $c\in C^{+}(F)$ of degree $\le t$, $\pi_{F,\BB}(c) = 0$ then any
$m\in \Mon_{t-2}$ satisfies \eqref{eqcom}, which in view of Theorem
\ref{thmcom} suffices to prove the theorem.

Let us first prove  that  \eqref{eqcom} holds for $m\in \BB_{t-2}$. We distinguish several cases.
If $x_im,x_jm\in \BB$ then (\ref{eqcom})  holds trivially.
Suppose next that $x_im,x_jm\in \partial\BB$. Then, $f:= x_im-\piFB(x_im)$ and $ f':=x_jm-\piFB(x_jm)$ belong to $F_{t-1}$.
As $x_j\gamma(f)=x_i\gamma(f')$, $x_jf-x_if'\in C^+(F)$ and thus, by our assumption,
$\piFB(x_j f) = \piFB(x_if')$, which gives (\ref{eqcom}).
Suppose now that $x_im\in \partial\BB$ and $x_jm\in\BB$. As before 
$f=x_im-\piFB(x_im)\in F_{t-1}$. If $x_j \gamma(f)=x_ix_jm\in\BB$ then
$x_jf\in C^+(F)$ and thus $\piFB(x_jf)=0$ gives (\ref{eqcom}).
Otherwise, $x_ix_jm\in\partial \BB$ and let $f':=x_ix_jm-\piFB(x_ix_jm)\in F_t$.
Now, $x_j\gamma(f)=\gamma(f')$ implies $x_jf-f' \in C^+(F)$ and thus
$\piFB(x_jf-f')=0$ which gives again (\ref{eqcom}).
This shows (\ref{eqcom}) in the case when $m\in \BB_{t-2}$, and thus we have
\begin{equation}\label{eqonB}
\pi_{F,\BB}(x_{i_2} \pi_{F,\BB}(x_{i_{1}} b)) = \pi_{F,\BB}(x_{i_1} \pi_{F,\BB}(x_{i_{2}} b)) \ \text{ for all } b\in  \Span{\BB}_{t-2}.
\end{equation}

\ignore{
For any $m \in \BB_{t-2}$ and any $i_{1}\neq i_{2}$ such that
$x_{i_{1}}\,m \in \partial \BB$, $x_{i_{2}}\,m \in \partial \BB$,
there exists $f, f' \in F_{t-1}$ such that
$\gamma(f)=x_{i_{1}}\,m$,  $\gamma(f')=x_{i_{2}}\,m$.
Thus, we have $\pi_{F,\BB}(x_{i_{1}}\,m) = \gamma(f)-f$, $\pi_{F,\BB}(x_{i_{2}}\,m) \gamma(f')-f'$ and $C(f,f') = x_{i_{2}}\, f- x_{i_{1}}\, f' \in
\Span{\BB^{+}}_{t}$. Consequently, 
\begin{eqnarray*}
\lefteqn{\pi_{F,\BB}(C(f,f'))= \pi_{F,\BB}(x_{i_2}f-x_{i_1}f') }\\
& = &
\pi_{F,\BB}(x_{i_2}(f-\gamma(f))- x_{i_1}(f'-\gamma(f')))\\
& = & - \pi_{F,\BB}(x_{i_2} \pi_{F,\BB}(x_{i_{1}} m))
+\pi_{F,\BB}(x_{i_1} \pi_{F,\BB}(x_{i_{2}} m)),
\end{eqnarray*}
which shows the relation \eqref{eqcom} for all $m\in \BB_{t-2}$
with $x_{i_{1}}\,m \in \partial \BB$, $x_{i_{2}}\,m \in \partial \BB$.
A similar proof applies for $m\in \BB_{t-2}$ if $x_{i_{1}}\,m \in \BB
$ or $x_{i_{2}}\,m \in \BB$. 
This implies that for all $b \in \Span{\BB}_{t-2}$,
\begin{equation}\label{eqonB}
\pi_{F,\BB}(x_{i_2} \pi_{F,\BB}(x_{i_{1}} b)) = \pi_{F,\BB}(x_{i_1} \pi_{F,\BB}(x_{i_{2}} b)).
\end{equation}
}

Let us now consider $m\in \Mon_{t-2} \backslash \BB_{t-2}$. By definition $\pi_{F,\BB}(x_{i}\, m)= \pi_{F,\BB}(x_{i'}\, \pi_{F,\BB}(m'))$
for some $m' \in \Mon_{t-2}$ and $i'\in [1,n]$ such that  $x_{i}\, m
=x_{i'}\, m'$. If $i \neq i'$ there exists $m''\in \Mon_{t-3}$ such that 
$ m=x_{i'}\, m'',  m'=x_{i}\, m''$.  As $F$ is also complete in degree $t-1$
and for all $c\in C^{+}(F)$ of degree $\le t-1$, $\pi_{F,\BB}(c) = 0$,
by induction hypothesis we have  
$$ 
\pi_{F,\BB}(x_{i'}\, \pi_{F,\BB}(x_{i} m'')) = \pi_{F,\BB}(x_{i}\, \pi_{F,\BB}(x_{i'} m'')),
$$
so that $\pi_{F,\BB}(x_{i}\, m) = \pi_{F,\BB}(x_{i}\, \pi_{F,\BB}(m))$.
If $i=i'$, we have by definition $\pi_{F,\BB}(x_{i}\, m) = \pi_{F,\BB}(x_{i}\, \pi_{F,\BB}(m))$.

As $\piFB(m)=m$ for $m \in \BB_{t-2}$, we deduce that
\begin{equation}\label{eqclaim}
\piFB(x_im)=\piFB(x_i\piFB(m)) \ \text{ for all } m\in \Mon_{t-2}, \ i\in [1,n].
\end{equation}

Now, using (\ref{eqclaim}),
$\piFB(x_i\piFB(x_jm))$ is equal to $\piFB(x_i\piFB(x_j\piFB(m)))$
 which in turn is equal to 
$\piFB(x_j\piFB(x_i\piFB(m)))$ (using (\ref{eqonB})) and thus to 
$\piFB(x_j\piFB(x_im))$ (using again  (\ref{eqclaim})).
We can now apply Theorem \ref{thmcom} and conclude the proof.
\ignore{
Since $b=\pi_{F,\BB}(m)\in \Span{\BB}_{t-2}$, the relation \eqref{eqonB} implies that
\begin{eqnarray*}
{\pi_{F,\BB}(x_{i_2} \pi_{F,\BB}(x_{i_{1}} m))  }
 & = &  \pi_{F,\BB}(x_{i_2} \pi_{F,\BB}(x_{i_{1}} \pi_{F,\BB}(m)))\\
 & = &  \pi_{F,\BB}(x_{i_1} \pi_{F,\BB}(x_{i_{2}} \pi_{F,\BB}(m)))\\
 & = &  \pi_{F,\BB}(x_{i_1} \pi_{F,\BB}(x_{i_{2}} m)),
\end{eqnarray*}
which proves that the relations \eqref{eqcom} hold. 
By Theorem~\ref{thmcom}, we deduce that $\pi_{F,\BB}$ is the projection of $\PK_{t}$  on $\Span{\BB}_{t}$ along
$\SpanD{F}{t}$.
}
\end{proof}
 
\begin{theorem}[border basis, \cite{Mourrain2005}]\label{thmnfanyt} \label{CorDegIdeal}
Let $\BB\subset \Mon$ be connected to 1 and let $F$ be a rewriting family for $\BB$, complete in any degree. Assume that $\pi_{F,\BB}(c) = 0$ for all 
$c \in C^{+}(F)$.
Then $\BB$ is a basis of $\PK/(F)$, $\PK=\Span{\BB} \oplus (F)$,
 and
$(F)_{t}=\SpanD{F}{t}$ for all $t\in \oN$; the set $F$ is a \emph{border basis} of the ideal
$I=(F)$ with respect to $\BB$.
\end{theorem}
\begin{proof}
By Theorem \ref{thmnfdegd}, 
$\PK_{t}= \Span{\BB}_{t} \oplus \SpanD{F}{t}$ for all $t\in \NN$. This  implies that
$\PK= \Span{\BB}\oplus (F)$ and thus $\BB$ is a basis of  $\PK/(F)$.
Let us prove that $(F)_{t}=\SpanD{F}{t}$ for all $t \in\NN$.
Obviously, $\SpanD{F}{t}\subset (F)_t$. Conversely let $p\in (F)_t$. Then
$p=r+q$, where $r\in \Span{\BB}_t$ and $q\in \SpanD{F}{t}$. Thus $p-q\in (F)\cap\Span{\BB}=\{0\}$, i.e., $p=q\in \SpanD{F}{t}$. 
\ignore{
which shows that $\BB$ is a generating set of $\PK/(F)$.
If there is a non zero polynomial $p\in \BB\cap (F)$, then $p\in
\SpanD{F}{t_{0}}$ for some $t_{0}\in \NN$. It is in contradiction with the
previous relation. 
Thus, we have $\PK=\Span{\BB} \oplus (F)$ and as any polynomial of degree $t$
is reduced by $F$ to an element of degree $\le t$ in $\Span{\BB}$, we also
have 
\[
\PK_{t}= \Span{\BB}_{t} \oplus \SpanD{F}{t} = \Span{\BB}_{t} \oplus (F)_{t}.
\]
As $\SpanD{F}{t}\subset (F)_{t}$, we deduce that the two spaces
supplementary to $\Span{\BB}_{t}$ are equal: $\SpanD{F}{t}=(F)_{t}$.
}
\end{proof} 

This implies the following characterization of border bases using 
the commutation property.
\begin{corollary}[border basis, \cite{m-99-nf}]\label{thmcomanyt} 
Let $\BB\subset \Mon$ be connected to 1 and let $F$ be a rewriting family for $\BB$, complete in any degree.
If for all $m\in \BB$ and all indices $i,j\in[1,n]$, we have:
\begin{equation*}
   \pi_{F,\BB}(x_i\,\pi_{F,\BB}(x_j\,m))=\pi_{F,\BB}(x_j\,\pi_{F,\BB}(x_i\,m)),
\end{equation*}
then $\BB$ is a basis of $\PK/(F)$, $\PK=\Span{\BB} \oplus (F)$, and
$(F)_{t}=\SpanD{F}{t}$ for all $t  \in \oN$.
\end{corollary}
\begin{proof}
Same proof as for Theorem \ref{thmnfanyt}, using Theorem \ref{thmcom}.\end{proof}
\ignore{
We check that for all $m\in \BB$,
$\pi_{F,\BB}(x_i\,\pi_{F,\BB}(x_j\,m))-\pi_{F,\BB}(x_j\,\pi_{F,\BB}(x_i\,m))$
implies that for all $ f, f' \in F$ such that $C(f,f') \in \Span{\BB^{+}}$
$\pi_{F,\BB}(C(f,f')) = 0$. 
By theorem \ref{thmnfanyt}, we deduce that $F$ is a border basis of $(F)$ for $\BB$.
\end{proof}
}
 

\section{Hankel Operators}\label{sec:3}
In this section, we analyse the properties of Hankel operators and related
moment matrices, that we will need hereafter, for the moment matrix
approach.
\subsection{Linear forms on the polynomial ring }

The set of $\oK$-linear forms from $\PK$ to $\oK$ is denoted by $\PK^{\ast}:=\tmop{Hom}_{\oK} (\PK,
\oK)$ and
called the dual space of $\PK$. A typical element of $\PK^{\ast}$ is the
evaluation at a point $\zeta \in \oK^n$:
\begin{eqnarray*}
  \tmmathbf{1}_{\zeta} & : & p \in \PK \mapsto p (\zeta) \in \oK.
\end{eqnarray*}
Such evaluation can be composed with differentiation. Namely, for
$\alpha\in\oN^n$, the differential functional:
\begin{eqnarray*}
\mathbf{1}_{\zeta}\cdot \partial^\alpha & :  & p \in \PK \mapsto
\left(\frac{\partial^{|\alpha|}}{\partial x_1^{\alpha_1} \ldots \partial x_n^{\alpha_n}}p\right)(\zeta)
\end{eqnarray*}
evaluates at $\zeta$ the derivative $\partial^\alpha$ of $p$. For
$\alpha=0$, $\un_{\zeta}\cdot \partial^0= \tmmathbf{1}_{\zeta}$. 
The dual basis of the monomial basis $(\xx^\alpha)_{\alpha \in \NN^{n}}$ of $\PK$ is denoted
$(\dd^\alpha)_{\alpha \in \NN^{n}}$; we have $\dd^\alpha(\xx^\beta) =\delta_{\alpha,\beta}$.
In characteristic $0$, $\dd^\alpha:= \un_{0} \cdot \frac{1}{\prod_{i=1}^n \alpha_i!} \partial^\alpha$.
Any element $\Lambda \in \PK^{\ast}$ can be written as $ \Lambda = \sum_{\alpha} \Lambda
(\xx^{\alpha})\dd^{\alpha}$. In particular,
$\tmmathbf{1}_{\zeta}=\sum_{\alpha \in \NN^{n}} \zeta^\alpha \dd^\alpha$.

For $S\subset \PK$, we define
\[ 
S^{\bot} := \{\Lambda \in \PK^{*} \mid \forall p \in S\ \Lambda(p)=0 \}.
\]

\subsection{Hankel operators}
The dual space $\PK^{\ast}$ has a natural structure of $\PK$-module
which is defined as follows: $(p,\Lambda)\in \PK\times \PK^{\ast}\mapsto
p\cdot\Lambda\in \PK^{\ast}$, where
\begin{eqnarray*}
  p \cdot \Lambda & : & q \in \PK \mapsto \Lambda (p q) \in \oK.
\end{eqnarray*}
Note that, for any $\alpha,\beta\in\oN^n$, we have
\begin{eqnarray*}
\xx^\beta \cdot \dd^\alpha & 
= & \dd^{\alpha-\beta}  \ \tmop{if}  \alpha \ge \beta,\\
  & = & 0 \ \ \ \ \ \ \ \tmop{otherwise}.
\end{eqnarray*}
\begin{definition}
For $\Lambda \in \PK^{\ast}$, the Hankel operator $H_{\Lambda}$ is the
operator from $\PK$ to $\PK^{\ast}$ defined by
\begin{eqnarray*}
  H_{\Lambda} & : & p \in \PK \mapsto p \cdot \Lambda \in \PK^{\ast} .
\end{eqnarray*}
\end{definition}

\begin{lemma}
For $\Lambda \in \PK^{\ast}$, the matrix of the Hankel operator
$H_\Lambda$ with respect to the bases $(\xx^\alpha)$ of $\PK$ and
$(\dd^\beta)$ of $\PK^{\ast}$ is
$[H_{\Lambda}]=(\Lambda(\xx^{\alpha+\beta}))$.
\end{lemma}

\begin{proof}
Writing $\Lambda=\sum_\gamma \Lambda(\xx^\gamma) \dd^\gamma$, we
have: \[H_\Lambda(\xx^\alpha)= \xx^\alpha\cdot \Lambda
= \sum_\gamma \Lambda(\xx^\gamma)\ \xx^\alpha \cdot  \dd^\gamma
= \sum_{\gamma\mid \gamma \ge \alpha} \Lambda(\xx^\gamma) \dd^{\gamma-\alpha}
= \sum_\beta \Lambda(\xx^{\alpha+\beta}) \dd^\beta.
\]
\end{proof}

\ignore{
\begin{remark}\label{rem::bilinearform}
Given $\Lambda \in \PK^{\ast}$, consider  the quadratic form
$Q_{\Lambda}$ defined  on $\PK$ by
\begin{eqnarray*}
  Q_{\Lambda} & : & (p, q) \in \PK^2 \mapsto \Lambda (p q) \in \oK.
\end{eqnarray*}
The matrix of $Q_{\Lambda}$ in the monomial basis $(\xx^{\alpha})$
of $\PK$ is $[Q_{\Lambda}] = (\Lambda (\xx^{\alpha + \beta}))$.
Moreover, 
$Q_{\Lambda} (p, q) = \Lambda(p q)=  H_{\Lambda} (p) (q) = H_{\Lambda} (q) (p)$ for all
$p,q\in \PK$. In the
following, we will also identify $H_{\Lambda}$ and $Q_{\Lambda}$.
\end{remark}
}

\medskip
We  now summarize some well known properties of the kernel
\[\Ker H_\Lambda=\{p\in \PK\mid p\cdot \Lambda =0,\ \text{ i.e., } \Lambda(pq)=0\ \forall q\in \PK\}.\]
of the Hankel operator $H_\Lambda$.
Recall the definition of a Gorenstein algebra \cite{DiEm05:cox}, \cite[Chap. 8]{em-07-irsea}.
\begin{definition}
An algebra $\AAA$ is called {\em Gorenstein} if $\AAA$ and its dual
space $\AAA^\ast$ are isomorphic $\AAA$-modules.
\end{definition}

Applying this definition to $\AAA:=\PK/\Ker H_\Lambda$ yields
\begin{lemma} \label{lem::ideal}
$\Ker H_\Lambda$ is an ideal in $\PK$ and the  quotient space
$\AAA:=\PK/\Ker H_\Lambda$ is a Gorenstein algebra.
\end{lemma}
\begin{proof}
Direct verification, using $H_\Lambda$ as isomorphism in the proof
of the second part of the lemma.
\end{proof}

The focus of this paper is the computation of zero-dimensional
varieties, which relates to finite rank Hankel operators as shown in
the following lemma.
\begin{lemma}
  \label{lem:basis}
The rank of the operator $H_{\Lambda}$ is finite if and
only if $\Ker H_\Lambda$ is a zero-dimensional ideal, in which case
$\dim \PK/\Ker H_\Lambda=\rank H_{\Lambda}$.
\end{lemma}
\begin{proof}
Directly from the fact that, given $p_1,\ldots,p_r\in\PK$,
$H_\Lambda(p_1),\ldots,H_\Lambda(p_r)$ are linearly independent in 
$\PK^\ast$ if and only if 
the cosets $[p_1],\ldots,[p_r]$ are linearly independent in
$\PK/\Ker H_\Lambda$.
\ignore{
We have $\dim (\PK/\Ker H_\Lambda) = \rank H_\Lambda =r$, thus
$\Ker H_\Lambda$ is a zero-dimensional ideal.
}
\end{proof}
The next theorem states a fundamental result in commutative algebra,
namely that all zero-dimensional polynomial ideals can be characterized using
differential operators (see \cite[Chap. 7]{em-07-irsea}, \cite[Thm. 2.2.7]{DiEm05:cox}). For the
special case of zero-dimensional Gorenstein ideals, a single
differential form is enough to characterize the ideal.

\begin{theorem}\label{thm::idealC}
Let $\oK=\oC$ and  assume $\rank H_{\Lambda} = r < \infty$. Then there
exist  $\zeta_1, \ldots, \zeta_d \in \oC^n$ (with $d\le r$) and
non-zero (differential) polynomials $p_1,\ldots,p_d\in \oC[\partial]$,
of the form  $p_i (\partial)= 
\sum_{\alpha\in A_i}a_{i,\alpha}\partial^\alpha$ where $A_i\subset \oN^n$
is finite and $a_{i,\alpha}\in\oK$, such that
  \begin{equation}
\Lambda=\sum_{i=1}^d  \un_{\zeta_{i}} \cdot p_{i}(\partial).
    \label{eq:lambda}
  \end{equation}
\end{theorem}



For a zero-dimensional ideal $I\subset \PK$ with simple zeros 
$V(I)= \{\zeta_1,\ldots,\zeta_r\} \subset \oK^n$ only, we have
$I^{\bot} = \Span{ \tmmathbf{1}_{\zeta_1},\ldots,\tmmathbf{1}_{\zeta_r}}$ and the ideal $I$
is radical as a consequence of Hilbert's Nullstellensatz.

In a similar way, we can now characterize the linear forms $\Lambda$
for which $\Ker H_\Lambda$ is a radical ideal.

\begin{proposition}\label{propradideal}
Let $\oK=\oC$ and assume that $\rank
H_{\Lambda} = r < \infty$. Then, the ideal $\Ker H_\Lambda$ is radical
if and only if
\begin{equation}\label{eqlambdarad}
\Lambda=\sum_{i=1}^r \lambda_i\tmmathbf{1}_{ \zeta_i} \ \text{ with }
\lambda_i\in \oK-\{0\} \ \text{ and } \zeta_i\in \oK^n \text{ pairwise
distinct},
\end{equation}
in which case $\Ker H_\Lambda=I(\zeta_1,\ldots,\zeta_r)$ is the
vanishing ideal of the $\zeta_i$'s.
\end{proposition}

\begin{proof}
Assume first that $\Ker H_\Lambda$ is radical with
$V(\Ker H_\Lambda):=\{\zeta_1,\ldots,\zeta_r\}\subset \oK^n$. This
implies $\Ker H_\Lambda=I(V(\Ker H_\Lambda))= I(\zeta_1,\ldots,\zeta_r)$. Let $p_i\in\oC[\xx]$ be interpolation
polynomials at the points $\zeta_i$, i.e., $p_i(\zeta_j)=\delta_{i,j}$ for
$i,j\le r$. Then the set $\{p_1,\ldots,p_r\}$ is linearly
independent in $\AAA:=\PK/(\Ker H_\Lambda)$ and thus is a basis of $\AAA$. As the linear
functionals $\Lambda$ and $\sum_{i=1}^r\Lambda(p_i)
\tmmathbf{1}_{\zeta_i}$ take the same values at each $p_i$, we obtain:
$\Lambda=\sum_{i=1}^r \Lambda(p_i) \tmmathbf{1}_{\zeta_i}$. Moreover,
$\lambda_{i}:=\Lambda(p_i)\ne 0$, since $\rank H_\Lambda = r$.

Conversely assume that $\Lambda$ is as in (\ref{eqlambdarad}). The
inclusion $I(\zeta_1,\ldots,\zeta_r)\subset \Ker H_\Lambda$ is
obvious. Consider now $p\in \Ker H_\Lambda$ and as before let
$p_i\in \PK$ be interpolation polynomials at the $\zeta_i$'s.
Then $0=\Lambda(p\,p_i)= \lambda_i p(\zeta_i)$ implies $p(\zeta_i)=0$, thus
showing $p\in I(\zeta_1,\ldots,\zeta_r)$. As $\Ker
H_\Lambda=I(\zeta_1,\ldots,\zeta_r)$ is the vanishing ideal of a set of $r$
points, it is radical by the Hilbert Nullstellensatz.
\end{proof}

In a similar way, we can also characterize real radical ideals using
Hankel operators.

\begin{proposition}\label{proprealradical}
Let $\oK = \oR$ and assume that $\rank H_\Lambda =r<\infty$.
Then, the ideal $\Ker H_\Lambda$ is real radical if and only
if
\begin{equation}\label{eqlambdarealrad}
\Lambda=\sum_{i=1}^r \lambda_i\tmmathbf{1}_{\zeta_i} \ \text{ with } \lambda_{i}
\in \oR-\{0\}\  \text{and}\ \zeta_i
\in \oR^n \text{ pairwise distinct}.
\end{equation}
\end{proposition}
\begin{proof}
If $\Ker H_\Lambda$ is real radical then $V(\Ker
H_\Lambda)=\{\zeta_1,\ldots,\zeta_r\}\subset \oR^n$, so that
\eqref{eqlambdarad} gives \eqref{eqlambdarealrad}. 
Conversely, if $\Lambda$ is as in (\ref{eqlambdarealrad}),  then
$\Ker H_\Lambda$ is real radical, since $\sum_j q_j^2 \in
\Ker H_\Lambda$ implies $\sum_jq_j(\zeta_i)^2=0$ and thus
$q_j(\zeta_i)=0$, giving $q_j\in \Ker H_\Lambda$.
\end{proof}

Let us now recall a direct way to compute the radical of the ideal $\Ker H_\Lambda$. 
First, consider  the quadratic form
$Q_{\Lambda}$ defined  on $\PK$ by
\begin{eqnarray}\label{rem::bilinearform}
  Q_{\Lambda} & : & (p, q) \in \PK^2 \mapsto \Lambda (p q) \in \oK.
\end{eqnarray}
Then, 
$Q_{\Lambda} (p, q) = \Lambda(p q)=  H_{\Lambda} (p) (q) = H_{\Lambda} (q) (p)$ for all
$p,q\in \PK$, and the 
matrix of $Q_{\Lambda}$ in the monomial basis $(\xx^{\alpha})$
 is $[Q_{\Lambda}] = (\Lambda (\xx^{\alpha + \beta}))$.
We saw in Lemma \ref{lem::ideal} that the algebra $\AAA=\PK/\Ker H_\Lambda$ is Gorenstein. An alternative characterisation of Gorenstein algebras states that 
the above quadratic form $Q_\Lambda$ defines a non-degenerate inner product on $\AAA$
(see eg. \cite{em-07-irsea}[chap. 9]). Assume now that $\rank H_\Lambda=r<\infty$ so that $\dim \AAA=r$.
Let $b_{1},\ldots, b_{r}$ be a basis of $\AAA$ and let
$d_{1},\ldots, d_{r}$ be its dual basis in $\AAA$  for
$Q_{\Lambda}$: it satisfies $\Lambda( b_{i} \, d_{j})= \delta_{i,j}$ for $i,j\in [1,r]$.
Then, for any element $a\in \AAA$, we have
\begin{equation} \label{eqdualbasis}
a =\sum_{i=1}^{r} \Lambda(a\, d_{i}) b_{i}.
\end{equation}
In particular, we have the following property:
\begin{proposition}\label{prop:trace} Let $\Delta :=\sum_{i=1}^{r} b_{i}\, d_{i}$. Given
 $h\in \AAA$, let $\XX_h$ be the corresponding multiplication operator
in $\AAA$. We have
$$ 
\Trace (\XX_{h})=\Lambda( h \Delta).
$$
\end{proposition}
\begin{proof} 
By relation \eqref{eqdualbasis}, the matrix of $\XX_{h}$ in the basis
$(b_{i})_{i\le i\le r}$ of $\AAA$ is $(\Lambda(h \, b_{j}\,
d_{i}))_{1\le i,j \le  r}$ and thus its trace is
$$ 
\Trace (\XX_{h})= \sum_{i=1}^{r} \Lambda( h \, b_{i}\, d_{i}) =\Lambda( h \Delta).
$$
\end{proof}
 
As a direct consequence we deduce the following result (see e.g., \cite{JFSMR08}):
\begin{theorem}\label{prop:radideal}
Let $\oK=\oC$ and assume $\tmop{rank} H_{\Lambda} = r < \infty$.
Let $b_{1},\ldots, b_{r}$ be a basis of $\AAA_{\Lambda}$, 
$d_{1},\ldots, d_{r}$ be its dual basis with respect to the inner product given by  $Q_{\Lambda}$, and  $\Delta
=\sum_{i=1}^{r} b_{i}\, d_{i}$. 
Then the radical of $\Ker H_\Lambda$ is  $\Ker H_{\Delta\cdot\Lambda}$.
\end{theorem}
\begin{proof}
Let $I:=\Ker H_\Lambda$. 
 A polynomial  $h$ is in $\sqrt{I}$ if and only if some  power of $h$ is in $I$ or, equivalently,
if and only if  $\XX_{h}$ is nilpotent. By a classical algebraic
property, the latter is equivalent to
$\Trace(\XX_{h}\XX_{a})=0=\Trace(\XX_{h\, a})$ for all $a \in  \AAA$.
Indeed, as the operators $\XX_{h}, \XX_{a}$ commute, if $\XX_{h}$ is nilpotent
then so is $\XX_{h}\XX_{a}$ and we have $\Trace( \XX_{h} \XX_{a})=0$.
Conversely if $\Trace(\XX_{h}
\XX_{a})=0$ for all $a \in \AAA$ then, by Cayley-Hamilton identity, the
characteristic polynomial $\det(\lambda I - \XX_{h})$ of $\XX_{h}$ is
$\lambda^r$ and thus $\XX_{h}$ is nilpotent.
By Proposition \ref{prop:trace}, we deduce that $h \in \sqrt{I}$ if and only if
$\Lambda(\Delta\, h a)=0$ for all $a \in \AAA_{\Lambda}$, that is, if and only if  $h\in
\Ker H_{\Delta\cdot\Lambda}$. 
\end{proof}

\subsection{Positive linear forms}

We now assume that $\oK=\oR$ and consider the polynomial ring $\PR$.
We first show that the kernel of a Hankel operator $H_\Lambda$ is  a
real radical ideal when $\Lambda\in \PR^{\ast}$ is positive. This
result is crucial in the algorithm that computes the real radical of
an ideal.

\begin{definition}
  We say that $\Lambda \in \PR^{\ast}$ is positive, which we denote
  $\Lambda \succcurlyeq 0$, if $\Lambda (p^2) \geqslant 0$ for all
  $p\in \PR$. Equivalently, we will say $H_{\Lambda} \succcurlyeq 0$
  if $\Lambda\succcurlyeq 0$.
\end{definition}

We will use the following simple observation.

\begin{lemma}\label{lemp2}
Assume  $\Lambda\succcurlyeq 0$.
For $p\in \PR$,
 $\Lambda (p^2) = 0$ implies $p \in \Ker H_{\Lambda} $ and thus
 $\Lambda (p) = 0$.
For $\Lambda,\Lambda'\succcurlyeq 0$,
$\Ker H_{\Lambda+\Lambda'}=\Ker H_\Lambda \cap \Ker H_{\Lambda'}$. 
\end{lemma}

\begin{proof}
For any $q\in \PR$, $t\in\oR$, $\Lambda((p+tq)^2) = t^2 \Lambda(q^2) +
2t \Lambda(p\,q) \ge 0$. Dividing by $t$ and letting $t$ go to zero
yields $\Lambda (p\,q)=0$, thus showing $p\in\Ker H_\Lambda$.
The inclusion $\Ker H_\Lambda \cap \Ker H_{\Lambda'} \subset \Ker H_{\Lambda+\Lambda'}$ is obvious.
Conversely, let $p\in \Ker H_{\Lambda+\Lambda'}$. In particular,
 $(\Lambda+\Lambda')(p^2)=0$, which implies $\Lambda(p^2)=\Lambda'(p^2)=0$ (since 
$\Lambda(p^2),\Lambda'(p^2)\ge 0$) and thus $p\in \Ker H_\Lambda \cap \Ker H_{\Lambda'}$.
\end{proof}

\begin{proposition}
  \label{prop:radreal}
If  $\Lambda \succcurlyeq 0$, then $\Ker H_{\Lambda}$ is a
real radical ideal.
\end{proposition}

\begin{proof}
Assume $\sum_i p_i^2\in \Ker H_\Lambda$; we show that
$p_i\in\Ker H_\Lambda$. Indeed, $(\sum_ip_i^2)\cdot \Lambda=0$
implies, for all $q\in \PR$, $0=\Lambda(\sum_ip_i^2q^2)
=\sum_i\Lambda(p_i^2q^2)$ and thus $\Lambda(p_i^2q^2)=0$. By Lemma
\ref{lemp2}, this in turn implies $\Lambda(p_iq)=0$ and thus $p_i\in
\Ker H_\Lambda$.
\end{proof}

We saw in Proposition {\ref{proprealradical}} that the kernel of
 a finite rank Hankel operator $H_\Lambda$ is real radical if and only if
$\Lambda$ is a linear combination of evaluations at real points. We
next observe that $\Lambda$ is positive precisely when $\Lambda$ is
a {\em conic} combination of evaluations at real points.

\begin{proposition}
  \label{proppositive}
Assume $\tmop{rank} H_{\Lambda} = r < \infty$. Then $\Lambda
\succcurlyeq 0$ if and only if $\Lambda$ has a decomposition
(\ref{eqlambdarealrad}) with $\lambda_i>0$ and distinct $\zeta_i\in
\oR^n$,
in which case $V(\Ker H_{\Lambda}) =\{\zeta_1, \ldots, \zeta_r
\}\subset \oR^n$.
\end{proposition}

\begin{proof}
If $\Lambda=\sum_{i=1}^r\lambda_i \tmmathbf{1}_{\zeta_i}$ with
$\lambda_i>0$ and $\zeta_i\in\oR^n$, then $\Lambda \succcurlyeq 0$
holds obviously. Conversely, assume that $\Lambda \succcurlyeq 0$ then
by Proposition \ref{prop:radreal} the ideal $\Ker H_\Lambda$ is
real radical. By Proposition \ref{proprealradical},
$\Lambda$ has a decomposition (\ref{eqlambdarealrad}) where
$\lambda_i=\Lambda(p_i)\ne 0$, $\zeta_i\in\oR^n$, and $p_i$ are
interpolation polynomials at the $\zeta_i$'s. As $p_i^2-p_i\in
I(\zeta_1,\ldots,\zeta_r)=\Ker H_\Lambda$, we have
$\Lambda(p_i)=\Lambda(p_i^2)\ge 0$, which concludes the proof.
\end{proof}

To motivate the next section, let us recall Lemma~\ref{lem:basis}
and observe how it specializes to truncated Hankel operators defined
on subspaces of $\PK$:
\begin{lemma} Let $\BB=\{b_{1},\ldots, b_{r}\} \subset \PR$ and $\Lambda \in \PK^{*}$.
The operator
\begin{align*}
H_{\Lambda}^\BB: \Span \BB &\rightarrow \Span{\BB}^\ast \\
               p = \sum_{i=1}^r  \lambda_i b_i & \mapsto  p \cdot \Lambda
\end{align*}
has a trivial kernel if and only if  the cosets $[b_1], \ldots, [b_r] \in \PK/\Ker H_{\Lambda}$
are linearly independent in $\PK/\Ker H_\Lambda$.
\end{lemma}
\begin{proof}
Direct verification using the fact that $\Ker H_{\Lambda}^\BB=\Ker H_{\Lambda}
\cap \Span{\BB}$.
\ignore{
Let $\{[b_1], \ldots, [b_r]\}$ be linear independent cosets in
$\PK/\Ker H_\Lambda$ and let $q = \sum_{j=1}^r\lambda_jb_j\in \Ker H_{\Lambda}^\BB$. Then, for
all $i=1,\ldots,r$, $\Lambda(b_i\,(\sum_{j=1}^r \lambda_j b_j))=0$.
This implies $\Lambda((\sum_{j=1}^r\lambda_jb_j)\,p)=0$ for all $p\in
\PK$ and thus $\sum_{j=1}^r\lambda_jb_j\in \Ker H_\Lambda$, implying
$\lambda_i=0$ for all $i=\{1,\ldots,r\}$. Conversely assume that
there are $\lambda_i \in \oK$ not all $\lambda_i = 0$ for which
$\sum_{j=1}^r\lambda_jb_j\in \Ker H_\Lambda$. That implies $\{[b_1],
\ldots, [b_r]\}$ not linearly independent in $\PR/\Ker H_\Lambda$.
}
\end{proof}

{Assume now  that $\Ker H_\Lambda$ is zero-dimensional and that $\BB
= \{b_1,\ldots,b_r \}\subset \PK$ is chosen so that $[b_1],\ldots,[b_r]$ form a basis of
$\AAA = \PR / \Ker H_\Lambda$.
As in relation (\ref{rem::bilinearform}), we consider the quadratic form
$Q^{\BB}_{\Lambda}$ on $\AAA$ defined by
\begin{eqnarray*}
  Q^{\BB}_{\Lambda} & : & (p, q) \in \AAA \times \AAA  \mapsto \Lambda (p q) \in \oK.
\end{eqnarray*}
Note that a matrix representation of this form can be obtained by
taking the principle submatrix of $[H_\Lambda]$ indexed by $\BB$.
Following \cite{Lau07}, we recall under which conditions the
bilinear form $Q^\BB_{\Lambda}$ relates to the Hermite form
\begin{align*}
T_{h}: \AAA \times \AAA &\rightarrow  \oK\\
(f, g) &\mapsto \Trace(\XX_{fgh})
\end{align*}
for some $h \in \AAA$.

\begin{lemma}
The quadratic form associated to $Q^{\BB}_\Lambda$ coincides with
the Hermite form $T_h$ for some $h \in \AAA$ if and only if $\Ker H_\Lambda$
is radical.
\end{lemma}
\begin{proof}
See \cite[Sec. 2.2]{Lau07}.
\end{proof}
 

\section{ Truncated Hankel Operators}\label{sec:4}

We have seen in the previous section that the kernel of the Hankel
operator associated to a positive linear form is a real radical
ideal. However, in order to be able to exploit this property into an
algorithm, we need to restrict our analysis to matrices of finite
size.  For this reason, we consider here truncated Hankel operators,
which will play a central role for the construction of (real) radical
ideals.

For $E\subset \PK$, set $E\cdot E:=\{p\,q\mid p,q\in E\}$. Suppose now
$E\subset \PK$ is a vector space.
A linear form $\Lambda$ defined on $\Span{E\cdot E}$ yields the map $\HE : E \rightarrow
E^{\ast}$ by $\HE(p) = p \cdot \Lambda$ for $p\in E$.  Thus $\HE$ can be seen
as a truncated Hankel operator, defined only on the subspace $E$. 

Given a subspace $E_0\subset E$, $\Lambda$ induces a linear map on $\Span{E_0\cdot E_0}$ and we can consider
the induced truncated Hankel operator
$H^{E_0}_\Lambda: E_0 \rightarrow (E_0)^\ast$.
\begin{definition}
\label{defflatextension}Given vector subspaces 
$E_0 \subset E \subset \PK$ and $\Lambda \in { \Span{E \cdot E}}^{\ast}$, $H_{\Lambda}^E$ is
 said to be a \textit{flat extension} of its restriction $H^{E_0}_{\Lambda}$ to
 $E_0$ if $\rank H^E_{\Lambda} = \rank H^{E_0}_{\Lambda}$.
\end{definition}

We now give some conditions ensuring that it is possible to
construct a flat extension of a given truncated Hankel operator.
The next result extends an earlier result of Curto-Fialkow
\cite{CF96}; a generalization of this result can be found in \cite{BBCM11}.

\begin{theorem}\cite{MoLa2008}
 \label{theoflatextension}
Consider a vector subspace $E \subset {\PK}$
 and a linear function $\Lambda$ on $\SpanK{E^+ \cdot E^+}$. 
Assume that $E=\SpanK{\CC}$ where $\CC\subset \Mon$ is connected to 1 and that
 $\rank H^{E^+}_{\Lambda} = \rank H^E_{\Lambda}$. 
Then there exists a
 (unique) linear function $\tilde{\Lambda} \in \PK^{\ast}$ which extends
 $\Lambda$, i.e., $\tilde{\Lambda} (p) = \Lambda (p)$ for all $p \in
 \SpanK{E^+ \cdot E^+}$, and satisfying $\rank H_{\tilde{\Lambda}} = \rank
 H^{E^+}_{\Lambda} .$ In other words, the truncated Hankel operator
 $H^{E^+}_{\Lambda}$ has a (unique) flat extension to a (full) Hankel operator
 $H_{\tilde{\Lambda}}$.
\end{theorem}

In the following, we will deal with linear forms vanishing on a given set $G$ of polynomials.
\begin{definition}
Given a vector space $E\subset \PK$ and $G\subset \SpanK{E\cdot E}$, define the set
\begin{equation}\label{eqKE}
\Lc_{G,E}:=\{\Lambda\in  \SpanK{E\cdot E}^{\ast}\mid \Lambda(g)=0\ \forall g\in G\}.
\end{equation}
If $\oK=\oR$, define 
\begin{equation}\label{eqKEpos}
\Lc_{G,E,\succeq}:=\{\Lambda\in \Lc_{G,E}\mid \Lambda(p^2)\ge 0\ \forall p\in E\}.
\end{equation}
For an integer $t\in \NN$ and $G\subset \PK_{2t}$, taking
$E=\PK_{t}$, we abbreviate our notation and set
$\Lc_{G,t}:=\Lc_{G,\PK_{t}}$ and $\Lc_{G,t,\succeq}:=\Lc_{G,\PK_{t},\succeq}$
when $\oK=\oR$.
\end{definition}

\subsection{Truncated Hankel operators and  radical ideals}

In this section, we assume that $E$ is a finite dimensional vector space. 
The following definition for generic elements of $\Lc_{G,E}$ is justified 
by Theorem \ref{theogeneric} below.

\begin{definition}\label{defgeneric}
Let $G\subset \SpanK{E\cdot E}$ where $E$ is a finite dimensional subspace of $\PK$.
An element $\Lambda^*\in \Lc_{G,E}$ is said to be {\em generic} if 
\begin{equation}\label{eqgenericE0}
{\rank}\,  H^{E_0}_{\Lambda^{*}}=\max_{\Lambda\in \Lc_{G,E}}
\rank H^{E_0}_{\Lambda}
\end{equation}
for all subspaces  $E_0\subset E.$
\end{definition}

If $\oL$ is a field containing $\oK$, we denote by $\Lc_{G,E}^{\oL}:=\Lc_{G,E}\otimes \oL$, the space obtained by considering the vector spaces 
over $\oL$ in \eqref{eqKE}.
We recall here a classical result about generic properties over field
extensions, which will be used to give a simpler proof of a result
that we need from \cite{LLR08b}.
\begin{lemma}\label{lem:generic}
Let $\oK$ be a field of characteristic $0$ and $\oL$ a field containing $\oK$. If 
$\Lambda^{*}$ is a generic element in $\Lc_{G,E}^{\oK}$, then it is generic in $\Lc_{G,E}^{\oL}$.
\end{lemma}
\begin{proof} 
The space of matrices $H^{E}_{\Lambda}$ for $\Lambda \in \Lc_{G,E}^{\oK}$
is a vector space spanned by a basis $H_{1},
\ldots, H_{l}$ over $\oK$ (resp. $\oL$). Let $u_{1},\ldots,u_{l}$ be new
variables and $\rho$ be the maximal size of a non-zero minor $\in \oK[\ub]$ of
$H(\ub):=\sum_{i=1}^l \, u_{i} H_{i}$. Then for any value of 
$\ub\in \oK^{l}$ (resp. $\ub\in \oL^{l}$), the matrix $H(\ub)$ is of rank
$\le \rho$. Since $\oK$ is of characteristic $0$ there exists $\ub_{0}\in
\oK^{l}$ with $H(\ub_{0})$ of rank $\rho$, which corresponds to a generic element
in $\Lc_{G,E}^{\oK}$ and in $\Lc_{G,E}^{\oL}$.
\end{proof}
 
\begin{theorem}\label{theogeneric}
Let $E$ be a finite dimensional subspace of $\PK$ and let $G\subset \SpanK{E\cdot E}$.
Assume $\Lambda^{*}\in \Lc_{G,E}$ is generic, ie. satisfies (\ref{eqgenericE0}).
Then, $\Ker H_{\Lambda^{*}}^{E} \subset \sqrt {(G)}.$
\end{theorem}
\begin{proof} 
By Lemma \ref{lem:generic}, $\Lambda^{*}$ is a generic
element of $\Lc_{G,E}$ over $\oR$ or $\oC$ and thus we can assume that $\oK=\clK$.
Let $v\in V_{\clK}(G)$, let $\unb_v$ denotes the evaluation at $v$ restricted
to $\Span{E\cdot E}$ and let $f\in \KE$. Our objective is to show that $f(v)=0$. 
Suppose for contradiction that $f(v)\neq 0$. 

Notice that $\unb_v$ and $\Lambda':=\Lambda^{*}+\unb_v$ belong to $ \Lc_{G,E}$.
As $\Lambda'(f^{2})=f^{2}(v)\neq 0$, $f\in \Ker H^E_{\Lambda}\setminus\Ker H^E_{\Lambda'}$ and
by the maximality of the rank of $\HE$ $\Ker H^E_{\Lambda'}\not\subset \Ker \HE$. Hence there exists $f'\in
\Ker H^E_{\Lambda'}\setminus \Ker \HE$. Then, 
$0= H_{\Lambda'}^{E} (f')= \HE(f')+f'(v)\unb_v$ implies $f'(v)\neq
0$. On the other hand, 
$$
0=H_{\Lambda'}^{E}(f')(f)=\Lambda'(ff')=\Lambda(ff')+f(v)f'(v)= H_{\Lambda^{*}}^{E}(f)(f') +f(v)f'(v)= f(v)f'(v),
$$ yielding a contradiction.
\end{proof}

\subsection{Truncated Hankel operators, positivity and real radical ideals}

We first give a result which relates the kernel of $\HE$ with the
real radical of an ideal $(G)$, when $\Lambda$ is positive and 
vanishes on a given set $G$ of polynomials.
We start with the following result, which motivates our definition of
the generic property for a positive linear form.
\begin{proposition}\label{lemkergensubset}
For $\Lambda^{*}\in \Lc_{G,E,\succeq}$, the following assertions are equivalent:
\begin{itemize}
 \item[(i)]  $\rank H_{\Lambda^{*}}^{E}= \max_{\Lambda \in \Lc_{G,E,\succeq}} \rank H_{\Lambda}^{E}$.
 \item[(ii)] $\Ker H_{\Lambda^{*}}^E \subset \Ker\HE$ for all $\Lambda \in \Lc_{G,E,\succeq}$.
\item[(iii)] $\rank  H_{\Lambda^{*}}^{E_0}=\max_{\Lambda \in \Lc_{G,E,\succeq}}
\rank H_{\Lambda}^{E_0}$ for any subspace $E_0\subset E$.
\end{itemize}
Call $\Lambda^{*}\in \Lc_{G,E,\succeq}$ {\em generic} if it satisfies any of the equivalent conditions (i)--(iii) and set
$$
\Kc_{G,E,\succeq} := \Ker H_{\Lambda^{*}}^E \ \text{ for any generic }\Lambda^{*}\in \Lc_{G,E,\succeq}.$$
\end{proposition}
\begin{proof}
(i) $\Longrightarrow$ (ii): Note that $\Lambda+\Lambda^{*}\in \Lc_{G,E,\succeq}$ 
and $\Ker H^E_{\Lambda+\Lambda^{*}}=\Ker H^E_\Lambda \cap \Ker H^E_{\Lambda^{*}}$ (using Lemma \ref{lemp2}). Hence,
$\rank  H^E_{\Lambda+\Lambda^{*}} \ge \rank  H^E_{\Lambda^{*}}$ and
thus equality holds. This implies  that
$\Ker H^E_{\Lambda+\Lambda^{*}}=\Ker H^E_{\Lambda^{*}}$ is thus contained in 
$\Ker H^E_\Lambda$.\\
(ii) $\Longrightarrow$ (iii): Given $E_0\subset E$, we show that
$\Ker H^{E_0}_{\Lambda^{*}}\subset \Ker H^{E_0}_{\Lambda}$. By Lemma \ref{lemp2}, we have
$\Ker H^{E_0}_{\Lambda^{*}}\subset \Ker H^{E}_{\Lambda^{*}}$ and, by the above, we have
$\Ker H^{E}_{\Lambda^{*}} \subset \Ker H^E_\Lambda$.\\
The implication (iii) $\Longrightarrow$ (i) is obvious.
\end{proof}
\begin{lemma}\label{lemsubsetker}
let $G_0\subset G \subset \Span{E\cdot E}$. Then,
$\Kc_{G_0,E, \succcurlyeq}\subset \Kc_{G,E, \succcurlyeq}$. 
\end{lemma}
\begin{proof}
Let $\Lambda \in \Lc_{G, E,\succeq}$ be a generic element, so that
$\Ker \HE= \Kc_{G,E, \succcurlyeq}$.
Obviously, $\Lambda\in \Lc_{G_0,E,\succcurlyeq}$, which implies that 
$\Ker \HE \supseteq \Kc_{G_0,E, \succcurlyeq}$.
\ignore{
 As $F\subset G$, we have $\Lc_{G, E,\succcurlyeq} \subset \Lc_{F, E,\succcurlyeq}$. Therefore, 
for a generic $\Lambda \in \Lc_{G, E,\succcurlyeq}\subset \Lc_{F, E,\succcurlyeq}$ such that $\Kc_{G,E,\succcurlyeq} =\Kc_{E}(\Lambda)$,
we have using Proposition \ref{lemkergensubset}
\[ 
\Kc_{F,E,\succcurlyeq} \subset \Kc_{E}(\Lambda) = \Kc_{G,E,\succcurlyeq}.
\]
}
\end{proof}

\begin{theorem}\label{theogenericpositive}
Let $G\subset \Span{E\cdot E}$, where  $E$ is a finite dimensional
subspace of $\PR$. 
 Then, 
 $\Kc_{G,E,\succcurlyeq} \subset \sqrt[\oR] {(G)}.$
\end{theorem}

\begin{proof}
Let $\Lambda $ be a generic element of $ \Lc_{G,E,\succcurlyeq}$, so that
$\Kc_{G,E,\succcurlyeq} =\Ker H^E_{\Lambda}$,  and let 
$v\in V_\oR(G)$; we show that
$\Ker H^E_{\Lambda} \subset I(v)$. As $\unb_v$, the evaluation at $v$
restricted to $\Span{E\cdot E}$, 
belongs to $\Lc_{G,E,\succcurlyeq}$, we deduce using Proposition \ref{lemkergensubset} that
$\Ker H^E_{\Lambda}  \subset \Ker H^E_{\unb_v} \subset I(v)$. 
This implies $\Ker H^E_{\Lambda}  \subset I(V_\oR(G))=\sqrt[\oR]G$.
\ignore{
we also have $\Lambda':=\Lambda+\Lambda_v\in \Lc_{G,E,\succcurlyeq}$.
Then $\KerE{\Lambda'}= \KE \cap \Ker
\KerE{\Lambda_v} \subset \KE$. Hence equality holds (by
the maximality of the rank of $\HE$) and thus 
$\KE\subset \KerE{\Lambda_v} \subset I(v)$, where the
last inclusion follows using the fact that $1\in E$.
}
\end{proof}

Given a subset $F\subset \PR$ and $t\in\NN$, consider  for $G$ the prolongation 
$\SpanD{F}{2t}$
of $F$ to degree $2t$, and the subspace $E=\PR_t$. For simplicity in the notation we set
\begin{equation}\label{eqKFt}
\Kc_{F,t,\succcurlyeq}:= \Kc_{\SpanD{F}{2t},\PR_t,\succcurlyeq},
\end{equation}
which is thus contained in $\sqrt[\oR]{(F)}$, 
by Theorem \ref{theogenericpositive}.
The next result (from \cite{LLR07}) shows that equality holds for $t$ large enough.
\ignore{
\begin{definition} 
Let $\Kc_{F,t}(\Lambda)=\Kc_{\SpanD{F}{2\,t},\PR_{t}}(\Lambda)$.
\end{definition}
}

\begin{theorem} \label{thmstablerad}\cite{LLR07}
Let $F\subset \PR$. There exists $t_{0}>0$ such that 
$(\Kc_{F,t,\succcurlyeq})=\sqrt[\oR]{(F)}$
for all $t\ge t_0$.
\end{theorem}
\noteML{Let's just give a reference, no proof.}
\ignore{
TO-BE-COMPLETED 
\note{This theorem also only holds for positive Hankel operators. For the general case we can only say that it converges to a maximum Gorenstein factor of $I$ (like in \cite{JFSMR08}).}
\end{proof}
}

\ignore{
An element $\Lambda\in \Lc_{G,E,\succeq}$ which satisfies these
properties is called a {\em generic element} of $\Lc_{G,E}$ (resp. $\Lc_{G,E,\succeq}$).
This implies that the kernels $\Kc_{E}(\Lambda)$ for all
the generic elements $\Lambda$ are identical. According to this proposition, we can define 
\begin{eqnarray*}
\Kc_{G,E} &:=&\Kc_{E}(\Lambda)\ \mathrm{for\ a\ generic}\ \Lambda \in
\Lc_{G,E}. \\
\Kc_{G,E,\succeq} &:=&\Kc_{E}(\Lambda)\ \mathrm{for\ a\ generic}\ \Lambda\in \Lc_{G,E,\succeq}.
\end{eqnarray*}
}

\section{Algorithm}\label{sec:5}
In this section, we describe the new algorithm to 
compute the (real) radical of an ideal. But before, we recall the graded moment
matrix approach for computing the real radical developed
in \cite{LLR08b}, and the border basis algorithm developed in \cite{Mourrain2005}. 

\subsection{The graded moment matrix algorithm}
In the graded approach, the following family of spaces is considered:
\begin{eqnarray*}
\Lc_{F,t,\succeq} &:= &\Lc_{\SpanD{F}{2\,t}, \PR_{t},\succeq} \\
 &= &\{\Lambda\in
\PR_{2\,t}\, \mid\, \forall f \in \SpanD{F}{2\,t}, \Lambda(f)=0
\, \mathrm{and}\, \forall p\in \PR_{t}, \Lambda(p^{2}) \geq 0 \}. 
\end{eqnarray*}
For $\Lambda \in \Lc_{F,t,\succeq}$, let $H_{\Lambda}^{t}:=
H_{\Lambda}^{\PR_{t}}$.

Algorithm \ref{algo:GRR} presents the graded moment matrix
algorithm described in \cite{LLR07}.

\begin{algorithm2e}[ht]\caption{\textsc{Graded Real Radical}}\label{algo:GRR}
\KwIn{a finite family $F$ of polynomials of $\PR$.}
Set $t:=1$ and $\delta=\max\{\deg(f),f\in F\}$;
\begin{enumerate}
 \item Choose a generic $\Lambda$ in $\Lc_{F,t,\succeq}$;
 \item Check wether $\rank H_{\Lambda}^{s} = \rank H_{\Lambda}^{s+1}$
   for some $s$ such that $\delta\leq s<t$;
 \item If not, increase $t:=t+1$ and repeat from step (1);
 \item Compute $\Ker H_{\Lambda}^{s}$; 
\end{enumerate}
\KwOut{$\sqrt[\oR]{(F)}=(\ker H_{\Lambda}^{s})$.}
\end{algorithm2e}

This algorithm requires in the first step to solve semi-definite
programming problems on
matrices of size the number of all monomials in degree $t$. This
number is growing very quickly with the degree when the number of
variables is important, which significantly slows down the performance of the
method when several loops are necessary.  The extension to compute the
radical is also possible with this approach by doubling the variables
and by embedding the problem over $\oC^{n}$ in $\oR^{2\,n}$. The
correctness of the algorithm relies on Theorem \ref{thmstablerad} 
which comes from \cite{LLR07}.

\subsection{The border basis algorithm}
Algorithm \ref{algo:BB} presents the border basis algorithm described
in \cite{Mourrain2005}. Hereafter, we analyze shortly the different steps.
\begin{algorithm2e}[ht]\caption{\textsc{Border Basis}}\label{algo:BB}
\KwIn{a family $F$ of polynomials of $\PK$.}
Set $t:=0$, $\BB:=\{1\}$, $G:=\emptyset$ and $\delta=\max\{\deg(f),f\in F\}$;
\begin{enumerate}
 \item Compute the reduction $\tilde{F}$ of $F_{t+1}$ on $\<\BB\>_{t+1}$ with
   respect to $G$;
 \item Set $t':=\min\{\deg(p), p\in \tilde{F}, p\neq 0\}-1$ ;
 \item Compute a minimal $\tilde{G}$ such that
   $\Span{\tilde{G}}:=\Span{G^{+},\tilde{F}} \cap \Span{\BB^{+}}_{t'+1}$; 
 \item Set $t''=\min\{\deg(p), p \in  {\tilde{G}}\cap \Span{\BB},
   p\neq 0\}-1$;
 Compute $\tilde{\BB}$ connected to $1$ 
such that $\Span{\BB^{+}}_{t''+1}:=\Span{\tilde{\BB}}_{t''+1}\oplus \Span{\tilde{G}}_{t''+1}$;
 \item Compute a rewriting family $G''$ of $\tilde{G}_{t''+1}$ with respect to
   $\tilde{\BB}_{t''+1}$;
 \item If $G''\neq G$ or $\tilde{\BB}\neq \BB$ or $t''<\delta$ then set $t:=t''+1$, $\BB:=\tilde{\BB}$, $G:= G''$ and
   repeat from step (1);
\end{enumerate}
\KwOut{the border basis $G$ of $(F)$ with respect to $\BB$.}
\end{algorithm2e}

In step (1), the reduction of a polynomial $p$ by a rewriting family
$G$ for a set $\BB$ consists of the following procedure: 
For each monomial $\xx^{\alpha}$ of the support of $p$ which is of the form 
$\xx^{\alpha} = x_{i}\, \xx^{\alpha'}\, \xx^{\alpha''}$ 
with $\xx^{\alpha'} \in \BB$ and $\xx^{\alpha''}$ of the smallest possible degree, 
if there exists an element $g=x_{i}\, \xx^{\alpha'} - r \in G$ with 
$r \in \Span{\BB}$, then the monomial $\xx^{\alpha}$ is replaced by $\xx^{\alpha''} \, r$.
This is repeated until all monomials of the remainder are in $\BB$.


Step (3) consists of the following steps:  take the coefficient matrix
$M=[M_{0}|M_{1}]$ of the polynomials in $G^{+} \cup \tilde{F}$ where
the block $M_{0}$ is indexed by the monomials in
$\partial \BB^{+}$ and the block $M_{1}$ is indexed by the monomials
in $\BB$ for a given ordering of the monomials, 
compute a row-echelon reduction $\tilde{M}$ of $M$,
and deduce the polynomials of $\tilde{G}$ corresponding to the non-zero rows of
$\tilde{M}$. For $p\in \tilde{G}$ corresponding to a non-zero row of
$\tilde{M}$, the monomial indexing its first non-zero coefficients is
denoted $\gamma(p)$. Notice that
$\Span{\tilde{G}}:=\Span{G^{+},\tilde{F}} \cap \Span{\BB^{+}}_{t'+1}$
contains the elements of $C^+(G_{t'})$.

Step (4) consists 
\begin{itemize}
\item of removing the monomials $\gamma(p)$ for $p\in  \Span{\tilde{G}_{t''+1}}\cap
\Span{\BB}_{t''+1}$, and 
 \item of adding the monomials in $\partial \BB \setminus
   \{\gamma(p)\mid 
 p \in \tilde{G}\}$ of degree $\leq t''+1$.
\end{itemize}

Step (5) consists of auto-reducing the polynomials $p\in \tilde{G}$ of degree $\leq
t''$ so that $\gamma(p)$ is the only term of $p$ in $\partial
\tilde{\BB}$. This is done by inverting the coefficient matrix of $\tilde{G}$ with
respect to the monomials in $\partial \tilde{\BB}$. Notice that as
$\Span{\BB^{+}}_{t''+1}:=\Span{\tilde{\BB}}_{t''+1}\oplus\Span{\tilde{G}}_{t''+1}$,
$\tilde{G}$ is complete in degree $t''+1$.

In step (6), if the test is valid then the loop start again with $G$ a
rewriting family of degree $t$
with respect to $\BB$, which is by definition included in
$\<\BB^{+}\>_{t}$. Thus, at each loop, $\tilde{G}$ contains $G$ 
and $C^{+}(G) \subset \Span{G^{+}}\cap
\Span{\BB^{+}}_{t+1}\subset \Span{\tilde{G}}$. 

The algorithm stops if $G''= G$ and $\tilde{\BB}= \BB$ and $t\ge \delta$. 
Then $t''=t$, $\tilde{G}=G$ and 
$C^{+}(G)\subset \tilde{G}=G$ is reduced to $0$ by $G$. If $G$ is a rewriting
family complete in degree $t$ for $\BB$, we deduce by Theorem \ref{thmnfdegd} that 
$\pi_{G,\BB}$ is the projection of $\PK_{t}$ on $\Span{\BB}_{t}$ along
$\SpanD{G}{t}$. As $\tilde{\BB}= \BB$, we also have $t\ge \max
\{\deg(b)\mid b \in \BB\}$ so that $G$
is a border basis with respect to $\BB$. As $t\ge \delta$, the elements of $F$ reduce
to $0$ by $G\subset F$. Thus $(G)=(F)$.

It is proved in \cite{Mourrain2005} that this algorithm stops when the ideal $(F)$ is
zero-dimensional. Thus its output $G$ is the border basis of the ideal $(F)$ with
respect to $\BB$.




\subsection{$\oK$-Radical Border Basis algorithm}
Our new radical border basis algorithm can be seen as a combination of the graded
real radical algorithm and the border basis algorithm. The 
modification of the border basis algorithm consists essentially of generating new
elements of the (real) radical of the ideal at each loop (step $(1')$ in Algorithm
\ref{algo:RBB}), and to use these new relations (which are in the
(real) radical by Theorem \ref{theogeneric} and Theorem \ref{theogenericpositive}) in step (3). 
In the case when $\oK=\oC$, a final stage is added to get the generators of the radical of a Gorenstein 
ideal (step (7) below).

\begin{algorithm2e}[ht]\caption{\textsc{$\oK$-Radical Border Basis}}\label{algo:RBB}
\KwIn{a family $F$ of polynomials of $\PK$.}

Set $t=0$, $\BB=\{1\}$, and $G=\emptyset$;
\begin{enumerate}
\item[($1'$)] {\em Compute 
a (maximal) $S \subset \BB_{t+1}$ such that $S\cdot S$ can be reduced by $G$ onto $\BB_{t+1}$
and $K :=$ \textsc{GenericKernel}$_{\oK}(G,\BB,S)$;}
 \item Compute the reduction $\tilde{F}$ of $F_{t+1}$ on $\<\BB^{+}\>_{t+1}$ with respect to $G$;
 \item {\em Set $t':=\min\{\deg(p), p\in \tilde{F} \cup K, p\neq 0\}-1$;}
 \item {\em Compute $\tilde{G}$ such that $\Span{\tilde{G}}:=\Span{G^{+},\tilde{F}, K} \cap \Span{\BB^{+}}_{t'+1}$;}
 \item Compute $\tilde{\BB}$ connected to $1$ and $t''\le t'$ maximal such that $\Span{\BB^{+}}_{t''+1}:=\Span{\tilde{\BB}}_{t''+1}\oplus \Span{\tilde{G}}_{t''+1}$;
 \item Compute a rewriting family $G''$ of $\tilde{G}_{t''+1}$ with respect to $\tilde{\BB}$;
 \item If $G''\neq G$ or $\tilde{\BB}\neq \BB$ or $t''<\delta$ then set $t:=t''+1$, $\BB:=\tilde{\BB}$, $G:= G''$ and
   repeat from step (1);
 \item {\em if $\oK = \oC$ then $[G,\BB] :=$ \textsc{Socle}$(G,\BB,\Lambda)$;}
\end{enumerate}
\KwOut{The border basis $G$ of the ideal $\sqrt[\oK]{(F)}$ with respect to $\BB$.}
\end{algorithm2e}
 
The two new ingredients that we describe below are the function 
\textsc{GenericKernel} (see Algorithm \ref{algo:GenericKernel}) used
to generate new polynomials in the (real) radical, and the function
\textsc{Socle} (see Algorithm \ref{algo:Socle}) 
 which computes the generators of the radical
from the border basis of a Gorenstein ideal when $\oK=\oC$.

\begin{definition}\label{def:fred}
Given a rewriting family $F$ with respect to $\BB$ and $S=\{\xx^{\beta_1},
\ldots, \xx^{\beta_l} \}$, we define $\Fred$ as the following family of polynomials :
For all $\xx^{\beta_i}, \xx^{\beta_j} \in S$ such that 
$\pi_{F,\BB} ( \xx^{\beta_i + \beta_j})$ exists and is in $\Span{S\cdot S}$, we define
$\kappa_{\beta_i + \beta_j}(\xx) = \xx^{\beta_i + \beta_j} - 
\pi_{F,\BB}(\xx^{\beta_i + \beta_j})$ and $\kappa_{\beta_i + \beta_j} = 0$
otherwise. 
\end{definition} 
With $\Fred$ as in Definition \ref{def:fred}, we are going to analyze the corresponding spaces $\Lc_{\Fred, S}$, 
$\Lc_{\Fred, S, \succeq}$, $\Kc_{\Fred, S}$, $\Kc_{\Fred, S,
  \succeq}$.  Notice that by construction $\Fred\subset\SpanD{F}{t}$
where $t= 2 \max \{ \deg(s) \mid s \in S \}$.

The construction of the generic kernel $\Kc_{\Fred,S}$ (resp., $\Kc_{\Fred,S,
  \succcurlyeq}$) is implemented by Algorithm
\ref{algo:GenericKernel}. 
This routine is the one that is executed for finding effectively new
equations in the (real) radical.

\begin{algorithm2e}[ht]\caption{\textsc{GenericKernel$_{\oK}(F,\BB,S)$}}\label{algo:GenericKernel}
\KwIn{A rewriting family $F$ with respect to $\BB$ allowing reduction for all the
  monomials in $S\cdot S$.}
\begin{enumerate}
\item If $\oK=\oC$, we construct an element $\Lambda \in \Lc_{\Fred, S}$
  such that $H_{\Lambda}^S$ has maximal rank,
by taking a generic element of the linear space $\Lc_{\Fred, S}$.
\item If $\oK=\oR$, we construct an element of $\Lambda \in \Lc_{\Fred, S,\succcurlyeq}$ such
that $H_{\Lambda}^S$ has maximal rank, by computing an element in
the relative interior of the feasible region
of the following {\em semi-definite programming problem}:
\begin{itemize}
\item[--]  $H= (h_{\alpha,\beta})_{\alpha,\beta\in S} \succcurlyeq 0$
\item[--]  $H$ satisfies the Hankel constraints
$h_{0,0}=1$, $h_{\alpha,\beta}=h_{\alpha',\beta'}$ if $\alpha+\beta=\alpha'+\beta'$.
\item[--]  $H$ satisfies the linear constraints $\sum_{\alpha} h_{\alpha}
  \kappa_{\beta,\alpha}=0$ 
for all $\beta \in S\cdot S$ such that $\kappa_{\beta} = \sum_{\alpha}
\kappa_{\beta,\alpha}\, \xx^{\alpha} \neq 0$.
\end{itemize}
\item Then we compute $K$ as a basis of the kernel of $H_{\Lambda}^{S}$.
\end{enumerate}
\KwOut{A family $K$ of polynomials in $\sqrt[\oK]{(F)}$.}
\end{algorithm2e}

Notice that primal-dual interior
 point solver implementing a self dual embedding do return such a solution
 automatically. For a remark on how to use other solvers, see \cite[Remark
   4.15]{LLR07}. 


\begin{algorithm2e}[ht]\caption{\textsc{Socle$(G,\BB,\Lambda)$}}\label{algo:Socle}
\KwIn{A border basis $G$ for $\BB$ connected to $1$ and $\Lambda \in \Span{\BB\cdot
  \BB}^{\ast}$ such that $H_{\Lambda}^{\BB}$ is invertible.}
\begin{enumerate}
\item Compute a dual basis of $\BB=\{b_{1},\ldots, b_{r}\}$ as follows:
$[d_{1},\ldots, d_{r}] = H^{-1} [b_{1},\ldots,
  b_{r}]$ where $H=(\Lambda(b_{i}b_{j}))_{1\le i,j \le r}$ is the
  matrix of  $H_{\Lambda}^{\BB}$;
\item Compute $\Delta =\sum_{i=1}^{r} b_{i}\, d_{i}$ and the matrix 
  $H_{\Delta}=(\Lambda(\Delta\, b_{i}\,b_{j}))_{1\le i,j \le r}$ by reduction of the
  elements $\Delta\, b_{i}\,b_{j}$ by $G$ to linear combinations of elements
  in $\BB$;
\item Compute $G'=\ker H_{\Delta}$ and apply the normal form algorithm to
  $G'\cup G$ in order to deduce a basis $\tilde{\BB}\subset \BB$ connected to $1$ and a border
  basis $G''$ for $\tilde{\BB}$ such that $(G'') = (G'\cup G)=\sqrt{F}$.
\end{enumerate}
\KwOut{A basis $\tilde{\BB}$ connected to $1$ and a border basis $G''$ of $\sqrt{(F)}$
  for $\tilde{\BB}$.}
\end{algorithm2e}

\section{Correctness of the algorithms}\label{sec:6}
In this section, we analyse separately the correctness of the algorithm 
over $\oR$ and $\oC$.
\subsection{Correctness for real radical computation}
We prove first the correctness of Algorithm \ref{algo:RBB} over $\oR$. 
\begin{lemma}\label{lemkerequiv} 
If $G$ is a rewriting family complete in degree $2\,t$ for $\BB$, then for $\Lambda\in
\SpanD{G}{2\,t}^{\bot}$
\[
\Ker H_{\Lambda}^{\PK_{t}} \equiv \Ker H_{\Lambda}^{\BB_{t}} \mod \SpanD{G}{t}.
\]
\end{lemma} 
\begin{proof}
For $\Lambda\in \SpanD{G}{2\,t}^{\bot}$, we have
\begin{eqnarray*}
p\in\Ker H_{\Lambda}^{\PK_{t}} &\Leftrightarrow & 
\Lambda(p\,q)=0\ \forall q\in \PK_{t} \\
 &\Leftrightarrow & \Lambda(b\,q)=0\ \forall q\in \PK_{t}\ 
\mathrm{where}\ b\in \Span{\BB}_{t}=p
\mod \SpanD{G}{t}\\
 &\Leftrightarrow & \Lambda(b\,b')=0\  \forall b\in \Span{\BB}_{t}\\
 &\Leftrightarrow & b\in \Ker H_{\Lambda}^{\BB_{t}}.
\end{eqnarray*}
Therefore 
\[\Ker H_{\Lambda}^{\PK_{t}} \equiv \Ker H_{\Lambda}^{\BB_{t}} \mod
\SpanD{G}{t},\]
which proves the equality of the two kernels modulo $\SpanD{G}{t}$.
\end{proof}

\begin{lemma}\label{lemsumdeg}
If $G$ is a rewriting family complete in degree $2\,t$ for $\BB$, such that 
$\PK_{2\,t}=\Span{\BB}_{2\,t} \oplus \SpanD{G}{2\,t}$, then 
\[ 
\SpanD{\Kc_{G,t,\succeq}}{t} \equiv \SpanD{\Kc_{\Gred,\BB_{t},\succeq}}{t}
 \mod \SpanD{G}{t}. 
\]
\end{lemma} 
\begin{proof}
Let $\Lambda\in \SpanD{G}{2\,t}^{\bot}$ be a generic element such that 
$\Kc_{G,t,\succeq} = \Ker H_{\Lambda}^{\PK_{t}}$. By Lemma
\ref{lemkerequiv} and Proposition~\ref{lemkergensubset}, we have  
\[
\Kc_{G,t,\succeq} = \Ker H_{\Lambda}^{\PK_{t}} \equiv \Ker H_{\Lambda}^{\BB_{t}} \mod \SpanD{G}{t}
\supset \Kc_{\BB_{t},\succeq} \mod \SpanD{G}{t}.
\]
Conversely, let $\Lambda\in \Span{\Gred}^{\bot}$ be a generic element such that 
$ \Kc_{G,\BB_{t},\succeq} = \Ker H_{\Lambda}^{\BB_{t}}$. As $\Span{\Gred}
\subset \SpanD{G}{2\,t}$, there exists
$\tilde{\Lambda}\in \SpanD{G}{2\,t}^{\bot}$ which extends $\Lambda$ to
$\PK_{t}$. Then we have  
\begin{eqnarray*}
\Kc_{G,\BB_{t},\succeq} &= &\Ker H_{\Lambda}^{\BB_{t}} = \Ker H_{\tilde{\Lambda}}^{\BB_{t}} \\
 &\equiv & \Ker H_{\tilde{\Lambda}}^{\PK_{t}} \mod \SpanD{G}{t}
\supset \Kc_{G,t,\succeq} \mod \SpanD{G}{t}. 
\end{eqnarray*}
\note{I don't understand this proof. Where do you address the fact that we
  may have $G^r \subsetneq G$?}
\noteBM{Split the proof in two lemmas; we use the reduction by $G$ ie. by 
  $\Gred \subset \SpanD{G}{t}$}
\end{proof}


\begin{lemma}\label{lemfixpt}  
If Algorithm~\ref{algo:RBB} terminates with outputs $G$ and $\BB$, then $(G)=\sqrt[\oR]{(F)}$
and $\BB$ is a basis of $\PK/\sqrt[\oR]{(F)}$.
\end{lemma}
\begin{proof}
If the algorithm stops, all boundary polynomials of $C^{+}(G)$ reduce to $0$ by
$G$. By Theorem \ref{thmnfdegd}, for all $t$ we have $\PK_{2\,t}= \Span{ \BB }_{2\,t} \oplus \SpanD{G}{2\,t}$. As 
$\Kc_{\Gred,\BB_{t},\succeq}=\{0\}$ by Lemma \ref{lemsumdeg}, we deduce that 
$$ 
\Kc_{G,t,\succeq} \subset \SpanD{G}{t}.
$$
By Theorem~\ref{thmstablerad}, there exists $s_{0}$ \note{But how do we know that we have gone far enough? We may terminate with a Border basis of some $J\subset\sqrt[\oR]{I}$.} such that 
\[ 
(\Kc_{F,s_{0},\succeq}) = \sqrt[\oR]{I},
\]
where $I=(F)$.
By lemma \ref{lemsubsetker}, for $t\ge  s_{0}$, 
$$ 
\Kc_{F,s_{0},\succeq}\subset \Kc_{G,t,\succeq} \subset
\SpanD{G}{t} \subset \sqrt[\oR]{I},
$$
which implies that $(G)= \sqrt[\oR]{I}$.
\end{proof}

\begin{proposition}
Assume that $V_{\oR} (F)$ is finite. Then the algorithm \ref{algo:RBB}
terminates. It outputs a border basis $G$ for $\BB$
connected to $1$, such that $\PR=\Span{\BB} \oplus (G)$ and $(G)=\sqrt[\oR]{I}$.
\end{proposition}
\begin{proof} First, we are going to prove by contradiction that when the number of real roots
is finite, the algorithm terminates.

Suppose that the loop goes for ever.  Notice that at each step either 
$G$ is extended by adding new linearly independent polynomials or
it moves to degree $t+1$. Since the number of 
linearly independent polynomials added to $G$ in degree $\le t$ is finite, there is a step
in the loop from which $G$ is not modified any more. In this
case, all boundary $C$-polynomials of elements of $G$ of degree $\le t$ 
are reduced to $0$ by $G_{t}$.  By Theorem \ref{thmnfdegd}, we have
\begin{equation}\label{eqdirectsum}
\PR_{t} = \Span{ \BB }_{t} \oplus \SpanD{G_{t}}{t}.
\end{equation}
We have assumed that the loop goes for ever, thus this property is true for any
degree  $t$. 
By Theorem~\ref{thmstablerad}, there exists $s_{0}$ such that 
\[ 
(\Kc_{F,s_{0}/2,\succeq}) = \sqrt[\oR]{I}.
\]
As any element of $\SpanD{F}{s_{0}}$ reduces to $0$ by the rewriting family
$G_{s_{0}}$, we have $\SpanD{F}{s_{0}}\subset \SpanD{G_{s_{0}}}{s_{0}}$.
By Lemma \ref{lemsubsetker}, we deduce that
\begin{eqnarray*}
 \Kc_{  F, s_{0}/2,\succeq} 
& \subset & \Kc_{ G_{s_{0}}, s_{0}/2 ,\succeq} .
\end{eqnarray*}
For a high enough number of loops, the set $G_{s_{0}}$ is not modified and we
have $\Kc_{G_{s_{0}}, \BB_{s_{0}/2},\succeq}=\{0\}$. Applying Lemma~\ref{lemsumdeg}
using Equation~\eqref{eqdirectsum}, we have
\[
\Kc_{G_{s_{0}}, s_{0}/2,\succeq} \subset \SpanD{G_{s_{0}}}{s_{0}}
\]
By construction $G_{s_{0}} \subset \sqrt[\oR]{I}$, thus
\[ 
(G_{s_{0}}) = \sqrt[\oR]{I}.
\]

Let $\BB_{0}\subset\PR$ which defines a basis in $\PR/\sqrt[\oR]{I}$ and of
smallest possible degree and let $d_{0}$ be the maximum degree of its
elements. Then any monomial $m$ of degree $d_{0}+1$ is equal modulo
$(\sqrt[\oR]{I})_{d_{0}+1}$ to an element $b$ in $\Span{\BB_{0}}$ of degree $\le
d_{0}$. 

By Theorem~\ref{CorDegIdeal},
\[
\SpanD{G_{d_{0}+1}}{d_{0}+1}  = (\sqrt[\oR]{I})_{d_{0}+1},
\]
thus $m-b \in \SpanD{G_{d_{0}+1}}{d_{0}+1}$ so that any monomial of degree
$d_{0}+1$ can be reduced by $G$ to a polynomial in $\PK_{d_{0}}$. Thus 
$\BB= \cap_{f\in G} (\gamma(f))^{c} \subset \PR_{d_{0}}$ is finite
and the algorithm terminates.

By Lemma \ref{lemfixpt}, the algorithm
outputs  a border basis $G$ with respect to $\BB$ 
connected to $1$,  such that $(G)=\sqrt[\oR]{I}$.
\end{proof}
 

\subsection{Correctness for the radical computation}
In this section, we show the correctness of the algorithm for radical
computation, that is with $\oK=\oC$.
\begin{proposition}
Assume that $V_{\oC} (F)$ is finite. Then the algorithm \ref{algo:RBB}
terminates and outputs a border basis $G$ for $\BB$
connected to $1$, such that $(G)=\sqrt{I}$ and
$\PC=\Span{\BB} \oplus \sqrt{I}$.
\end{proposition}
\begin{proof}
Since the family $G$ contains the polynomials constructed by the normal
form algorithm \cite{Mourrain2005} and as $V_{\oC} (I)$ is zero-dimensional,
the normal form algorithm terminates and so do algorithm \ref{algo:RBB}. When
the loop stops, all boundary polynomials of $C^{+}(G)$ for any degree reduce
to $0$ by $G$ and $\Kc_{\Gred,\BB}=\{0\}$. By Theorem \ref{thmnfanyt}, $G$
is a border basis with respect to $\BB$. Let $\Lambda \in
\Span{\BB\cdot \BB}^{\ast}$ such that $\Kc_{\Gred,\BB}= \Ker
H_{\Lambda}^{\BB}$. By definition of $\Lambda$ and normal form property, if $f
\in \Span{\BB\cdot \BB} \cap (G)$ then $f$ reduces to $0$ by $G$ and
$\Lambda(f)=0$. This shows that we can extend $\Lambda$ to $\tilde {\Lambda}
\in \PC^{\ast}$ by $\tilde{\Lambda}=\Lambda$ on $\Span{\BB}$ and
$\tilde{\Lambda}=0$ on $(G)$.  We deduce that $(G)=\Ker H_{\tilde{\Lambda}}$
and that $\AAA_{\Lambda} =\PC/\Ker H_{\tilde{\Lambda}} = \PC/(G)$ is
Gorenstein. Let $d_{1},\ldots, d_{r}$ be the dual basis of $\BB$ for
$Q_{\Lambda}$ and $\Delta=\sum_{i=1}^{r} b_{i}\, d_{i}$.  By Theorem
\ref{prop:radideal}, $\Ker H_{\Delta\cdot \Lambda}^{\BB}$ computed in the
function \textsc{Socle}, yields a new basis $\BB'$ connected 
  to $1$ and a new border basis $G'$ such that $(G')=\sqrt{I}$.
\end{proof}
 
\section{Examples}\label{sec:7}

This section contains two very simple examples which illustrate the effect
of the SDP solution in one loop of the Real Radical Border Basis algorithm.  
The results in the next example are coming from a \texttt{C++} implementation available in the
package \texttt{newmac} of the project {\sc mathemagix}. It uses a version of
\texttt{lapack} with templated coefficients and \texttt{sdpa}
\footnote{\texttt{http://sdpa.sourceforge.net/}} with
extended precision so that all the computation can be run with extended precision
arithmetic.   
\subsection{A univariate example}
We give here a simple example in one variable to show how the real roots can be
separated from the complex roots, using this algorithm. We consider the 
polynomial $f=x^{4}-x^{3}-x+1$ with a single real root $x=1$ of multiplicity $2$.
In the routine \textsc{GenericKernel} of the algorithm, a $4\times 4$ matrix $H$ is constructed and
the linear constraints deduced from the relations
$x^{4}\equiv x^{3}+x-1$,
$x^{5}\equiv x^{3}+x^{2}-1$,
$x^{6}\equiv 2\,x^{3}-1$ modulo $f$ imposes the following form:
$$ 
H :=\left(
\begin{array}{cccc}
1& a & b & c \\
a & b & c &c+a-1\\
b & c & c+a-1 & c+b-1\\
c & c+a-1 & c+b-1&2\,c -1
\end{array}
\right).
$$
where $a=\Lambda(x),b=\Lambda(x^{2}),c=\Lambda (x^{3})$. 
The SDP solver yields the solution 
$$ 
\left(
\begin{array}{cccc}
1& 1 & 1 & 1 \\
1& 1 & 1 & 1 \\
1& 1 & 1 & 1 \\
1& 1 & 1 & 1 
\end{array}
\right)
$$ 
which kernel is $\Span{x-1, x^{2}-x,
x^{3}-x^{2}}$. Thus the output of the algorithm is $(x-1)$ the real radical of
$(f)$, the basis $\BB=\{1\}$ and the real root $x=1$. 

\subsection{A very simple bivariate example}
Let $f_{1}=x^{2}+y^{2}$ and $F=\{f_{1}\} \subset \oR[x,y]$.  The
algorithm computes the following:
\begin{itemizeminus}
 \item $\BB=\Mon-(y^{2})$
 \item We compute \textsc{GenericKernel}  in degree $1$ by choosing
$S=\{1,x,y\}$ with $S\cdot S\supset \support\, f_{1}$.

The SDP problem to solve reads as follows: find $h
 =[a,b,c,d] \in \oR^4$ such that
\[H=\left(
\begin{array}{ccc}
1 & a & b \\
a & c & d  \\
b & d & - c \\
\end{array}
\right) \succcurlyeq 0\] and has of maximal rank. Here
$a=\Lambda(x),b=\Lambda(y),c=\Lambda(x^{2}), d=\Lambda(x\, y)$. The condition $H \succcurlyeq 0$
implies that
\begin{itemize}
 \item $c=0$,
 \item $a=0, b=0, d=0$.
\end{itemize}
and consequently that $\ker H=\Span{x,y}$. Thus $x, y$ are returned by \textsc{GenericKernel} and added to $F$.
 \item After one iteration  the border basis algorithm stops and we obtain $\BB=\{1\}$
and $\sqrt[\oR]{(x^{2}+y^{2})}=(x,y)$.
\end{itemizeminus}

\subsection{Numerical example}
The tables below compare the size of the SDP problems to solve in our approach and in the method
described in \cite{LLR07}.
The $degree$ indicates the degree in the loop of the Border Basis Real radical algorithm, 
$n.sdp$ is the size of matrices in the corresponding SDP problem and $n.constraints$ the 
number of linear constraints involved, $t$ is the degree of the relaxation problem in \cite{LLR07}
and $n.sdp\ grad.\ rel.$ the size of matrices in the corresponding SDP problem.
$$
\begin{array}{|c|c|c|c|c|}
\hline
\multicolumn{5}{|c|}{Katsura\ 4}\\
\hline
degree & n.sdp & n. constraints & t & n.sdp\ grad.\ rel.\\
\hline
2 & 5 & 5 & 2 & 56\\
4 & 11 & 67 & 2 & 56\\
6 & 16 & 176 & 2 & 56\\
\hline
\end{array}
$$ 
$$
\begin{array}{|c|c|c|c|c|}
\hline
\multicolumn{5}{|c|}{Katsura\ 5}\\
\hline
degree & n.sdp & n.constraints & t & n.sdp\ grad.\ rel.\\
\hline
2 & 6 & 6 & 3& 84\\
4 & 16 & 146 & 3& 84\\
6 & 26 & 479 & 3 & 84\\
\hline
\end{array}
$$ 
$$
\begin{array}{|c|c|c|c|c|}
\hline
\multicolumn{5}{|c|}{bifur}\\
\hline
degree & n.sdp & n. constraints & t & n.sdp\ grad.\ rel.\\
\hline
2 & 4 & 2 & 8 & 165\\
4 & 9 & 32 & 8 & 165\\
6 & 16 & 150& 8 & 165\\
8 & 25 & 446& 8 & 165\\
8 & 16 & 152& 8 & 165\\
8 & 16 & 153& 8 & 165\\
6 & 16 & 158& 8 & 165\\
6 & 16 & 162& 8 & 165\\
4 & 9 & 34& 8 & 165\\
6 & 16 & 168& 8 & 165\\
6 & 16 & 169& 8 & 165\\
4 & 9 & 36& 8 & 165\\
6 & 16 & 177& 8 & 165\\
4 & 4 & 3& 8 & 165\\
4 & 8 & 37& 8 & 165\\
\hline
\end{array}$$

The tables below give the time for computing the real radical with the
solvers \texttt{sdpa}\footnote{\texttt{http://sdpa.sourceforge.net/}} or
\texttt{csdp}\footnote{\texttt{https://projects.coin-or.org/Csdp/}}
integrated into the border basis algorithm available in the package
\texttt{newmac} of {\sc mathemagix}. 
$$
\begin{array}{|c|c|c|c|c|c|c|c|c|c|}
\hline
Example & T_{\oR} & \small Gen.\ Ker. & CSDP & SVD\ Drop & Deg  & Deg_{\mathbb{C}}& N_{\mathbb{R}} & N_{\mathbb{C}} & \mathrm{T}_{\oC}\\
\hline\hline
 \multicolumn{10}{|c|}{Precision\ 90}\\  
 \hline
 kat4 &22.479s & 22.281s & 22.1645 & 1e-12 & 4 & 4 & 12 & 16 & 0.06s\\
 \hline
 kat5 & 146.49s & 146.29s & 145.64s & 1e-12 & 5 & 5 & 16 & 32 & 0.165s\\
 \hline
 cyclo & 10.839s & 10.7646s & 10.6243s & 1e-20 & 5 & 5 & 4 & 16 & 0.03s\\
 \hline
 robot & 41.84s & 41.52s & 41.26s & 1e-19 & 6 & 8 & 4 & 40 & 1.3s\\
 \hline\hline
 \multicolumn{10}{|c|}{Precision\ 120}\\
 \hline
 kat4 &22.557s & 22.28s & 22.16 & 1e-14 & 4 & 4 & 12 & 16 & 0.06s\\
 \hline
 kat5 & 146.59s & 146.39s & 145.1s & 1e-12 & 5 & 5  & 16 & 32 & 0.17s\\
 \hline
 cyclo & 10.839s & 10.7646s & 10.6243s & 1e-20 & 5 & 5 & 4 & 16 & 0.03s\\
 \hline
 robot & 42.884s & 42.5216s & 42.2447s & 1e-19 & 6 & 8& 4 & 40 & 1.4s\\
 \hline
\end{array}
$$ 
{\em A precision of 90 or 120 bits is used during the computation
but unfortunately the SDP solver is very, very, very slow for this
precision. A strange behavior/bug of the parameter used in the relaxation of the
barrier function is observed. The solution of this problem is in progress.
}
$$
\begin{array}{|c|c|c|c|c|c|c|c|c|c|}
\hline
Example & T_{\oR} & \small Gen.\ Ker. & SDPA-GMP & SVD\ Drop & Deg  & Deg_{\mathbb{C}}& N_{\mathbb{R}} & N_{\mathbb{C}} & \mathrm{T}_{\oC}\\
\hline\hline
\multicolumn{10}{|c|}{Precision\ 90}\\
\hline
kat4 &4.18s & 4.12s & 3.26 & 1e-18 & 4 & 4 & 12 & 16 & 0.06\\
\hline
kat5 & 26.28s & 26.01s & 23.16s & 1e-18 & 5 & 5 & 16 & 32 & 0.165\\
\hline\
cyclo & 10,95 & 10.77 & 10.64 & 1e-20 & 5 & 5 & 4 & 16 & 0.03\\
\hline
robot & 19,84 & 19.52 & 19.26 & 1e-19 & 6 & 8 & 4 & 40 & 1.3s\\
\hline
\end{array}
$$

Using {\tt SDPA-gmp} as the solver allows us a great improvement in
efficiency though we expect futher improvements improving both the way
connection with {\tt SDPA-gmp} is operated and better tuning the parameters
{\tt SDPA}.

\end{document}